\documentclass[12pt]{article}
\usepackage{verbatim,amsmath,amssymb,mathrsfs,amsthm}
\usepackage{fancyhdr}
\usepackage{graphicx}
\usepackage{subfigure}
\usepackage{caption}
\usepackage{amsmath}
\usepackage{makecell}
\usepackage{multirow}
\usepackage{array}
\usepackage{setspace}
\usepackage{float}
\usepackage{color,mathtools}
\usepackage{epstopdf}
\usepackage{lipsum}
\usepackage{appendix}
\usepackage{algorithm}
\usepackage{algorithmic}
\usepackage{blindtext}
\usepackage[textwidth=15.8cm]{geometry}

\usepackage{amsthm,amsmath,amssymb}
\newtheorem{theorem}{Theorem}

\newtheorem{remark}{Remark}
\newtheorem{assumption}{Assumption}

\renewcommand{\thetable}{\thesection.\arabic{table}}
\renewcommand{\theequation}{\thesection.\arabic{equation}}
\renewcommand{\thefigure}{\thesection.\arabic{figure}}

\begin{document}


\title{Discontinuous Galerkin method based on the reduced space for the nonlinear convection-diffusion-reaction equation}
\author{Shijin Hou\thanks{School of Mathematical Sciences, University of Science and Technology of China, Hefei, Anhui
230026, People's Republic of China. Email: {\tt{houshiji@mail.ustc.edu.cn}}.}
\and Yinhua Xia\thanks{School of Mathematical Sciences, University of Science and Technology of China, Hefei, Anhui
	230026, People's Republic of China. Email: {\tt{yhxia@ustc.edu.cn}}. This author was partially supported by National Natural Science Foundation of China  grant No. 12271498}}
\date{}
\maketitle

\begin{abstract}
  In this paper, by introducing a reconstruction operator based on the Legendre moments, we construct a reduced discontinuous Galerkin (RDG) space that could achieve the same approximation accuracy but using fewer degrees of freedom (DoFs) than the standard discontinuous Galerkin (DG) space. The design of the ``narrow-stencil-based'' reconstruction operator can preserve the local data structure property of the high-order DG methods. With the RDG space, we apply the local discontinuous Galerkin (LDG) method with the implicit-explicit time marching for the nonlinear unsteady convection-diffusion-reaction equation, where the reduction of the number of DoFs allows us to achieve higher efficiency.  In terms of theoretical analysis, we give the well-posedness and approximation properties for the reconstruction operator and the $L^2$ error estimate for the semi-discrete LDG scheme. Several representative numerical tests demonstrate the accuracy and the performance of the proposed method in capturing the layers. 
\end{abstract}



\section{Introduction}
In this paper, we consider the following nonlinear unsteady convection-diffusion-reaction (CDR) equation
\begin{subequations}\label{CDR equation with initial solution}
	\begin{align}
	&u_t+\nabla\cdot(\pmb b(\pmb x) f(u))-\varepsilon\Delta u+r(u) = g(\pmb x,t),\ (\pmb x,t)\in \Omega\times (0,T]),\label{CDR equation}\\
	&u(\pmb x,0) = u_0(\pmb x),\ \pmb x\in \overline{\Omega},
\end{align}
\end{subequations}	
where the diffusion velocity $\varepsilon$ is positive, the convection velocity field $\pmb b(\pmb x)$, $f(u)$, $r(u)$ and $g(\pmb x,t)$ are smooth, the initial solution $u_0$ belongs to $L^2(\Omega)$, and $\Omega$ is a bounded domain in $\mathbb{R}^d$ with the dimension $d$. With various appropriate boundary conditions, it yields a well-posed problem. The CDR equation has always received considerable attention as a model for fluid flow and heat transfer problems. It is widely applied in various fields of science and engineering such as chemical process simulation, river pollution, reservoir simulation, financial problems, etc \cite{roos2008robust,safdari2015radial}. Among these applications, a challenging scenario arises when numerically solving a convection-dominated or reaction-dominated type of CDR equation (i.e., $\varepsilon\ll 1$) whose solution may suffer from sharp internal or boundary layers \cite{leveque2007finite}. 
In such cases, the standard finite element method leads to spurious numerical oscillations.
In order to overcome this, a number of stabilized numerical methods have been developed, such as the streamline upwind Petrov-Galerkin (SUPG) \cite{brooks1982streamline,baysal2012stabilized}, the Gaussian radial basis function (RBF) \cite{li2017h,rashidinia2018stable}, etc.

In recent decades, the discontinuous Galerkin (DG) methods have been proposed as a kind of robust, accurate method for numerically solving the convection-dominated problem and have received a lot of research \cite{cockburn2000development}.  It can capture the interior or boundary layers well thanks to two aspects: First, by adopting the discontinuous basis space, DG methods can more flexibly describe the complicated structure of the solution near the layers. Second, the numerical flux naturally guarantees the upwind property. In addition, DG methods have many advantages, such as high
parallel efficiency, easy implementation on complicated geometries, etc. Thus, a range of DG methods has been proposed for the CDR equation. Paul et al. developed the $hp$-version DG method for several second-order partial differential equations with nonnegative characteristic forms \cite{houston2002discontinuous}. Ayuso et al. applied the weighted-residual approach to derive the DG scheme for the steady state CDR equation in \cite{ayuso2009discontinuous}. Nguyen et al. presented the implicit high-order hybridizable DG methods for the time-dependent nonlinear convection-diffusion (CD) equations \cite{nguyen2009implicit}. Cockburn and Shu studied the Runge-Kutta discontinuous Galerkin (RKDG) method for the time-dependent convection-dominated parabolic problems \cite{cockburn2001runge}. As an extension of the RKDG method, they proposed the Local Discontinuous Galerkin (LDG) methods for nonlinear time-dependent CD systems \cite{cockburn1998local}. Xu et al. provided $L^2$ error estimates for the semi-discrete LDG method for nonlinear CD equations \cite{xu2007error}. Wang et al. analyzed the stability and error estimation of the implicit-explicit (IMEX) LDG method for multi-dimensional nonlinear CD equation \cite{wang2016local}. However, all the above DG methods suffer from considerable degrees of freedom (DoFs), which could lead to higher computational costs than the traditional finite element methods (FEMs).

Therefore, several improved methods have been developed to reduce the number of DoFs. Cockburn et al. developed a hybridizable DG method for the steady-state CDR equation \cite{cockburn2009hybridizable}.
Recently, Li et al. proposed a novel approach based on the patch reconstruction discontinuous Galerkin space, in which the arbitrary high-order DG methods have only one degree of freedom per element \cite{li2020least}. This method has been used to solve the steady-state CD equation in \cite{sun2020discontinuous}. Its excellent performance in reducing the number of DoFs is widely recognized. However, high-order reconstruction requires a wide stencil according to the strategy in \cite{li2020least}. In this case, the local data structure property of the DG methods is weakened to some extent.


In this article, we aim to propose an efficient DG method for the nonlinear unsteady CDR equation. Our first contribution is to propose a reconstruction operator using the Legendre moments \cite{yap2005efficient}. Here, we consider a fixed narrow stencil consisting only of the element itself and its direct neighbors, instead of the wide stencil used in \cite{li2020least}, and implement the high-order reconstruction by exploiting the high-order Legendre moments on each element of the stencil. In addition, we study the well-posedness and some approximation properties of this reconstruction operator for our later analysis. Applying this operator, we construct the reduced discontinuous Galerkin (RDG) space which is able to use reduced DoFs to achieve the same approximation accuracy as the standard DG space. It's worth noting that the high-order reconstruction approach based on a narrow stencil can well preserve the local data structure property of DG methods. In light of the above advantages, we can develop some efficient DG methods based on this RDG space.
In the second part, we apply the LDG method based on our RDG space for spatial discretization and provide the error estimation for the semi-discrete LDG scheme. Combined with the implicit-explicit Range-Kutta (IMEX RK) time discretization method, we propose the complete IMEX RK LDG method. In the IMEX RK scheme, the diffusion part is treated implicitly, which avoids the severe step restriction but introduces a large linear system to be solved at each time step, showing that our RDG space can effectively reduce the computational cost.

The paper is organized as follows. In Section \ref{Sec:Notation and preliminaries}, we introduce the notations, definitions, and preliminaries used later in the paper. In Section \ref{Sec:Reduced discontinuous Galerkin space}, we define the compact reconstruction operator and present some approximation properties of it. Using this operator, we also give the definition of the RDG space. In Section \ref{Sec:LDG scheme}, the IMEX RK LDG method is applied to solve the nonlinear CDR equation. In addition, the error estimate is derived in the $L^2$ norm. In Section \ref{Sec:Numerical examples}, we present numerical results for several one-and two-dimensional CDR equations to demonstrate the accuracy. Finally, we draw the conclusion in Section \ref{Sec:Conclusion}.

\section{Notation and preliminaries}\label{Sec:Notation and preliminaries}
For an open and bounded domain $K\subset\mathbb{R}^d$, let $W^{s,p}(K)$ with indexes $s\ge 0$, $1\le p\le \infty$ denote the Sobolev space of functions whose derivatives up to order $s$ belong to the space $L^{p}(K)$. Its seminorm and norm are denoted by $|\cdot|_{s,p,K}$ and $\Vert\cdot\Vert_{s,p,K}$ respectively. In the most common case, $p = 2$, we write $H^s(K)$ instead of $W^{s,2}(K)$ for simplicity and denote the corresponding seminorm and norm by $|\cdot|_{s,K}$ and $\Vert\cdot\Vert_{s,K}$. For $s = 0$, $H^s(K)$ coincides with $L^2(K)$. Hence the norm and inner product of $L^2(K)$ can be denoted by $\Vert\cdot\Vert_{0,K}$ and $(\cdot,\cdot)_K$. Moreover, Let $P^k(K)$ denote the space of polynomials of degree at most $k$ on $K$.

Consider a $d$-dimensional hypercube $\Omega = \prod_{i=1}^d [a_i,b_i]$ and let $\mathcal{T}_h$ denote the rectangular partition with disjoint $N$ elements. Assuming that we distribute $N_{i}$ elements on the $i$th direction, i.e., $[a_i,b_i] = \cup _{j = 1}^{N_i}I_j^i$, where $I^i_{j}=[x^i_{j-\frac{1}{2}},x^i_{j+\frac{1}{2}}]$. Thus we have $N = \prod_{i=1}^dN_i$. Here we define a set of multiple indicators, $\Theta^d=\{\theta = (\theta_{1},\cdots ,\theta_{d}){\big |}1 \le \theta_{i} \le N_{i},i = 1,\cdots,d\}$.
For each $\theta\in \Theta^d$, the $\theta$th element is denoted by $K_\theta = \prod_{i=1}^dI^i_{\theta_i}$.
The corresponding center point is denoted by $\pmb{x}_\theta = (x_{\theta}^1,\cdots,x_{\theta}^d)$ with $x_{\theta}^i = (x^i_{\theta_i-\frac{1}{2}}+x^i_{\theta_i+\frac{1}{2}})/2$. Let $\mathcal{E}_h$ denote the set of all edges of elements in $\mathcal{T}_h$, and $\mathcal{E}_h^0=\mathcal{E}_h\backslash\partial\Omega$ denote the set of interior edges. Moreover, for every $K\in \mathcal{T}_h$, we denote its area by $|K|$ and the element length in $i$th direction by $h_{K}^i$.
Naturally, we can define the maximum  and minimum mesh size by
\[
h = \max_{K\in \mathcal{T}_h,\ 1\le i\le d}h_K^i,\ \underline{h} = \min_{K\in \mathcal{T}_h,\ 1\le i\le d}h_K^i.
\]
We assume that $\mathcal{T}_h$ is regular: With mesh refinements, there always exists a real positive number $\gamma$ independent of $h$ such that the ratio of the maximum and the minimum mesh size can be bounded by $\gamma$, i.e., $h\le \gamma \underline{h}$.

Given the mesh $\mathcal{T}_h$, we denote by $Q^k(K)$ the tensor product
piecewise polynomials of degree at most $k$ in each variable on the element $K$. Then, a piecewise polynomial space can be defined as:
\[V_h^k = \{v\in L^2(\Omega):v|_{K}\in Q^k(K),\ \forall K\in\mathcal{T}_h\},
\]
which is the general discontinuous Galerkin space.

Based on the regularity assumption, there are several useful properties that will be used for the later analysis. For convenience to description, we denote by $C$ a positive constant which may depend on the regularity of the function and the regularity parameter $\gamma$ but is independent of $h$ and adopt it to represent all scaling constants with the same characters in this paper.

\paragraph{Agmon's inequality}For any $v\in H^1(K)$, there exists a positive constant $C$ such that
\begin{equation}\label{Agmon inequality}
	\Vert v \Vert^2_{0,\partial K}\le C\left(h^{-1}\Vert v \Vert^2_{0,K}+h| v |_{1,K}\right).
\end{equation}
\paragraph{Approximation property}For any $v\in H^{k+1}(K)$, there exists an approximation $v_h\in P^k(K)$ satisfying
\begin{equation}\label{Approximation property}
	\Vert v -v_h\Vert_{0,K}+h| v -v_h|_{1,K}\le Ch^{k+1}| v |_{k+1,K}.
\end{equation}
\paragraph{Inverse properties}For any $v_h\in V_h^k$, there exists a positive constant $C$ such that
\begin{subequations}\label{inverse properties}
	\begin{align}
		&\Vert \partial_xv_h \Vert_{0,K}\le Ch^{-1}\Vert v_h \Vert_{0,K},\\
		&\Vert v_h \Vert_{0,\partial K}\le Ch^{-\frac{1}{2}}\Vert v_h \Vert_{0,K},\\
		&\Vert v_h \Vert_{0,\infty,K}\le Ch^{-\frac{d}{2}}\Vert v_h \Vert_{0,K}.
	\end{align}
\end{subequations}
Agmon's inequality (\ref{Agmon inequality}) has been proved in \cite{da2014mimetic}. More details of inverse properties (\ref{inverse properties}) can be seen in \cite{ciarlet2002finite}.

\section{Reduced discontinuous Galerkin space}\label{Sec:Reduced discontinuous Galerkin space}
In this section, we introduce the compact reconstruction operator by employing the Legendre moments on a fixed narrow stencil for the cases $d=1,2$. Moreover, we give the well-posedness condition and some approximation properties for the reconstruction operator, which is critical to the error estimate later. In the end, we define the reduced discontinuous Galerkin (RDG) space by applying this reconstruction operator.

\subsection{Reconstruction operator}\label{sec:Compact reconstruction operator}
Let us start by introducing a fixed narrow element stencil. In the case of one dimension, we select the element itself and its two neighbor elements as the stencil. Given the mesh $\mathcal{T}_h = \{K_j\}_{j=1}^N$, the element stencil for periodic boundary value problems can be denoted as follows:
\begin{equation*}
	\begin{aligned}
   		&S^1(K_j) = \{K_{j-1},K_j,K_{j+1}\}, \ j = 1,\cdots,\ N,\\
	\end{aligned}
\end{equation*}
where $K_{0}, K_{N+1}$ are obtained by the periodic extension. For other non-periodic boundary conditions, we denote the following bias stencil
\begin{equation*}
	\begin{aligned}
		&S^1(K_1) = \{K_1,K_2,K_3\},\\
   		&S^1(K_j) = \{K_{j-1},K_j,K_{j+1}\}, \ j = 2,\cdots,\ N-1,\\
		&S^1(K_N) = \{K_{N-2},K_{N-1},K_{N}\}.
	\end{aligned}
\end{equation*}
For two dimensions, the stencil is given by the tensor product of the one-dimensional stencil defined above. Consider the mesh $\mathcal{T}_h = \{K_\theta|\ \theta\in \Theta^2\}$, where $K_\theta =I^1_{\theta_1}\times I^2_{\theta_2}$ as denoted in the previous section.
Here, we can define the one-dimensional stencil for the first direction by $\{S^1(I^1_{\theta_1})\}_{\theta_1=1}^{N_1}$, similarly, $\{S^1(I^2_{\theta_2})\}_{\theta_2=1}^{N_2}$ for the second direction. The two-dimensional stencil is defined by 
\[
S^2(K_\theta) = S^1(I^1_{\theta_1})\times S^1(I^2_{\theta_2}).
\]
In what follows, for any $K\in \mathcal{T}_h$, we denote the stencil by $S(K)$ without distinguishing one or two dimensions.

Given the standard Legendre basis functions $\{\widehat{L}^i(\widehat{x})\}_{i = 0}^k$ \footnote{The $k$th order Legendre polynomials are defined as: $\widehat{L}^k(\widehat{x}) = \frac{(-1)^k}{2^kk!}(\frac{d}{dx})^k[(1-x^2)^k]$ \cite{abramowitz1964handbook}.} for $\widehat{x}\in[-1,1]$, we denote the multiple dimensional Legendre basis functions of degree $(\sum_{i = 1}^d\alpha_i)$ on element $K$ by
\[
L_{K}^\alpha(\pmb x) = \prod_{i = 1}^d\widehat{L}^{\alpha_i}\left(\frac{2(x_i-x_{\theta_i})}{h_K^i}\right).
\]
Here $\alpha = (\alpha_1,\dots,\alpha_d)$ is a multiple indicator belonging to $A^k =  \{(\alpha_1,\cdots, \alpha_d)|\ 0\le\alpha_i\le k\}$.
For $\alpha \in A^k$ and the element $K$, the Legendre moments of order $(\sum_{i = 1}^d\alpha_i)$ for a function $u\in C^0(K)$ are defined as \cite{yap2005efficient}:
\[
I^{\alpha}_{K}u(\pmb x) = \left(\prod_{i = 1}^d\frac{2\alpha_i+1}{h_K^j}\right)\int_{K}L_{K}^\alpha(\pmb{x})u(\pmb{x})d\pmb{x}.
\]
These Legendre moments can provide sufficient information to determine an interpolating approximation for $u$. By simple computation, we can verify that the following polynomial $p(\pmb x)\in Q^k(K)$ with the Legendre series expansion form 
\[
p(\pmb x) = \sum_{\alpha\in A^k} I_K^\alpha u(\pmb x) \prod_{i = 1}^d\widehat{L}^{\alpha_i}\left(\frac{2(x_i-x_{\theta_i})}{h_K^i}\right)
\]
is a $(k+1)$th order approximation for $u$ in the $L^2$ norm. The coefficients of $p(\pmb x)$ are determined by solving an approximation problem as follows
\[
	I^\alpha_{K}p(\pmb x) = I^\alpha_{K}u(\pmb x),\ \forall \alpha\in A^k.
\]

We aim to define a local reconstruction operator $R_{K}^k$ which generates a $(k+1)$th order approximation polynomial $R_{K}^ku\in P^k(S(K))$ for a piecewise continuous function $u$.
Inspired by the above interpolation problem, we consider using several low-order Legendre moments on each element of the stencil $S(K)$ to determine the high-order approximation polynomial $R_{K}^ku$.
As a result, the local reconstruction operator $R_{K}^k$ can be defined by solving the following approximation problem: to find $R_{K}^ku\in Q^k(S(K))$ such that
\begin{equation}\label{system of local reconstruction operator}
	I^\alpha_{K_\theta}R_{K}^ku = I^\alpha_{K_\theta}u,\ \forall K_\theta\in S(K),\ \alpha\in A^m,\ m<k.
\end{equation}
Given the stencil, we can make sure the system \eqref{system of local reconstruction operator} is determined by increasing the order of Legendre moments, $m$.
In the next section, we will discuss in detail the relation of $m$ and $k$ and give the well-posedness of the defined reconstruction operator. The proposal of this reconstruction operator, which requires only a narrow stencil, is one of the main contributions of this paper.

After that, the global reconstruction operator $R^k: V_h^m\to V_h^k$ is naturally defined piecewise: let the restriction of global approximation $R^ku$ be the corresponding local approximation $R_{K}^ku$, i.e.,
\begin{equation}\label{global approximation operator}
(R^ku)|_{K} = R_{K}^ku,\ \forall K\in \mathcal{T}_h.
\end{equation}
We can observe that the global reconstruction operator $R^k$ embeds the piecewise polynomial space $V_h^m$ into a higher order piecewise polynomial $V_h^k$. Let the embedded space be denoted by $U_h^k = R^kV_h^m$.

\subsection{Properties for the reconstruction operator}\label{Properties for Compact reconstruction operator}
For a better application of the reconstruction operator $R^k$, we would like to analyze its properties. Let us begin with a problem left over from the last section: the well-posedness of the local reconstruction operator $R^k_K$. It is equivalent to the existence and uniqueness of the solutions for the approximation problem \eqref{system of local reconstruction operator}, which can be deduced from the following assumption directly.
\begin{assumption}\label{admissible assumption}
	For every $K\in \mathcal{T}_h$ and $w\in Q^k(S(K))$, if the Legendre moments satisfy that
	\begin{equation}\label{admissible assumption equation}
	I^\alpha_{K_\theta}w=0,\ \forall K_\theta\in S(K),\ \alpha\in A^m,
	\end{equation}
	we must have $w|_{S(K)}\equiv 0$.
\end{assumption}
For the cases of $k = 2,5$, the detailed numerical analysis is presented in Appendix \ref{The numerical analysis of well-posedness}. We can observe that Assumption \ref{admissible assumption} holds when satisfying $(k+1)/(m+1) = 3$, where $3$ is the width of the stencil. For instance, the reconstruction operator $R^5_{K}$ demands a set of Legendre moments $\{I^\alpha_{K_\theta}u\}$ with $\alpha\in A^1$ and $K_\theta\in S(K)$. In the following sections, we call the relation $(k+1)/(m+1) = 3$ well-posed condition. Remark that the same condition is expected to be derived for the higher-order cases.

Under Assumption \ref{admissible assumption}, we can define a kind of $l_\infty$ norm for any $u\in Q^k(S(K))$ based on the Legendre moments as follows,
\[
\Vert u\Vert_{l_{\infty}(S(K))} = \max_{K_\theta\in S(K),\alpha\in A^m}|I^\alpha_{K_\theta}u|.
\]
On one side, the equivalence of the norms over finite dimensional space $Q^k(S(K))$ leads to the following property
\begin{equation}\label{norms equivalence}
	\Vert u\Vert_{0,\infty,S(K)}\le C\Vert u\Vert_{l_{\infty}(S(K))},\ \forall u\in Q^k(S(K)).
\end{equation}
On the other side, based on the boundedness of Legendre functions $L_{K_\theta}^\alpha$, we can find an upper bound $C$ such that 
\[
\max_{K_\theta\in S(K),\alpha\in A^m}|I^\alpha_{K_\theta}1|\le C.
\]
Therefore, for any $u\in C^0(S(K))$, we have
\begin{equation}\label{upper bound for l_infty discrete norm}
	\Vert u\Vert_{l_{\infty}(S(K))}\le C\Vert u\Vert_{0,\infty,(S(K))}.
\end{equation}
With the help of this norm, in Appendix \ref{Proof of Theorem1}, we provide the proof of the following theorem where the approximation properties of the local reconstruction operator are presented.
\begin{theorem}\label{Theorem 1}
	If Assumption \ref{admissible assumption} holds,
	the $k$-exactness property holds for any $K\in \mathcal{T}_h$ as\\
	\begin{equation}\label{k-exactness property for local reconstruction operator}
		R^k_{K}u = u,\  \forall\ u\in Q^k(S(K_j)).
	\end{equation}
	Moreover, for any $K\in \mathcal{T}_h$ the local approximation $R^k_{K}u$ with $u\in H^{k+1}(S(K))$ satisfies the $L^\infty$ error estimate as
	\begin{equation}\label{l_infty error estimate for local reconstruction operator}
		\Vert u-R_{K}^ku\Vert_{0,\infty,K} \le Ch^{k+1-\frac{d}{2}},
	\end{equation}
	and the $L^2$ error estimate holds as follows
	\begin{equation}\label{l2 error estimate for local reconstruction operator}
		\Vert u-R_{K}^ku\Vert_{0,K} \le Ch^{k+1},
	\end{equation}
	\begin{equation}\label{sime norm error estimate for local reconstruction operator}
		|u-R_{K}^ku|_{1,K} \le Ch^k,
	\end{equation}
	where $C$ depends on $|u|_{H^{k+1}(S(K))}$ but is independent of $h$.
\end{theorem}

\subsection{The reduced discontinuous Galerkin space}
In this subsection, we would like to further explore the approximation space $U_h^k$. 
For all $K\in \mathcal{T}_h,\alpha\in A^m$, consider the following locally supported functions defined on the domain $\Omega$
\begin{equation}\label{local basis function for V}
	v_{K}^\alpha =
	\left\{
	\begin{aligned}
		&L_{K}^\alpha(\pmb x),\ \pmb x\in K,\\
		& 0,\ {\rm elsewhere}.
	\end{aligned}
	\right.
\end{equation}
They span the piecewise polynomial space $V_h^m$ completely. By applying the reconstruction operator $R^k$ on each function $v_{K}^\alpha$, we can define the basis function of space $U_h^k$ by
\begin{equation}\label{basis function of RDG space}
\phi_{K}^\alpha = R^kv_{K}^\alpha.
\end{equation}
Therefore, any approximation function $R^ku\in U_h^k$ can be expressed by the basis expansion:
\[
R^ku = \sum_{K\in \mathcal{T}_h,\alpha\in A^m}(I_{K}^\alpha u)\phi_{K}^\alpha.
\]

By the definitions \eqref{global approximation operator}, \eqref{local basis function for V}, and \eqref{basis function of RDG space}, we know that the basis function $\phi_{K}^\alpha$ belongs to the piecewise polynomial space $V_h^k$ and has a compactly supported set $S(K)$. With the relation of $m$ and $k$ introduced in Section \ref{Properties for Compact reconstruction operator}, we can
conclude that the space $U_h^k$ of cardinality $N(m+1)^d$ certainly is the subspace of the general piecewise polynomial space $V_h^k$ and the reduction of $1/3^d$ degree of freedom can be appreciated.
According to \cite{hughes2000comparison}, the number of DoFs can be regarded as a proper indicator of the efficiency in a specific discrete system. From this point of view, the DG methods based on the RDG space can achieve higher efficiency. Moreover, our narrow-stencil reconstruction method ensures the local property of these basis functions, which prevents the local property of DG approaches from being destroyed. 

Up to this point, we have defined an approximation space $U_h^k$ with reduced cardinality but the same accuracy as the standard DG space $V_h^k$, called the reduced discontinuous Galerkin (RDG) space. Considering the above properties of the RDG space, it is easy to implement various DG methods with our RDG space.

\section{The LDG method with the RDG space}\label{Sec:LDG scheme}
\subsection{The semi-discrete LDG scheme}
Here, we only focus on the periodic boundary condition for simplicity. Notice that the analysis can be extended to other non-periodic boundary conditions \cite{wang2018third}.  It is worth noting that the RDG space is not limited to the LDG method, but can be applied to other DG methods as well.

To describe the LDG method, we begin with the following first-order system equivalent to the CDR equation \eqref{CDR equation}:
\begin{equation*}
\left\{
\begin{aligned}
		&u_t+\nabla\cdot (\pmb b(\pmb x)f(u)-\sqrt{\varepsilon}\pmb q)+r(u) = g(\pmb x,t),\\
		&\pmb q = \sqrt{\varepsilon}\nabla u.
	\end{aligned}
	\right.
\end{equation*}
Given the rectangular mesh $\mathcal{T}_h$, the functions in the RDG space are piecewise continuous,
whose behavior on the set of edges $\mathcal{E}_h$ may be undefined. Here we introduce some definitions to handle this case. Let $e\in \mathcal{E}_h$ be the edge of any element $K\in \mathcal{T}_h$. For a scalar function $u_h\in U_h^k$, we denote its two traces on $e$ along the positive and negative direction of the coordinate axis by $u_h^+$ and $u_h^-$. The jump and the mean of function $u_h$ at edge $e$ are denoted by $[u_h] = u_h^+-u_h^-$ and $\bar{u}_h = \frac{1}{2}(u_h^++u_h^-)$. We denote by $\pmb n^+$ the unit normal at $e$ of $K$ along the positive direction of the coordinate axis and denote by $\pmb n^-$ its reverse. For a vector function $\pmb u_h\in (U_h^k)^d$, the jump of it is denoted by $[\pmb u_h] = \pmb u_h^+\cdot\pmb n^++\pmb u_h^-\cdot\pmb n^-$. Let $\pmb n = (n_1,\dots,n_d)$ denote the unit outward normal on the boundary $\partial K$, and $(\cdot,\cdot)_K$ and $\langle\cdot,\cdot\rangle_{\partial K}$ denote the inner product on element $K$ and boundary $\partial K$ respectively.

The semi-discrete LDG scheme can be defined: find $u_h\in U_h^k$ and $\pmb q_h = (q_h^1,\dots,q_h^d)$ $\in(U_h^k)^d$ such that, for any $t>0$ and $K\in\mathcal{T}_h$, we have
\begin{equation}\label{weak formulation(WF)}
	\left\{
	\begin{aligned}
		&((u_h)_t,v_h)_K=\mathcal{H}_K(u_h,v_h)+\mathcal{L}_K(\pmb q_h,v_h)+\mathcal{R}_K(u_h,v_h)+\mathcal{G}_K(\pmb x,t,v_h),\ \forall v_h\in U_h^k,\\
		&(q_h^i,p_h^i)_K = \mathcal{K}_K^i(u_h,p_h^i),\ \forall p_h^i\in U_h^k,\ i = 1,\dots, d.
	\end{aligned}
	\right.
\end{equation}
Here, all operators in the above weak formulation are defined as follows,
\begin{subequations}\label{CDR operators}
\begin{align}
\mathcal{H}_K(u_h,v_h) &= \left((\pmb b(\pmb x)f(u_h),\nabla v_h)_K-\langle\pmb b(\pmb x)\hat{f}(u_h)\cdot \pmb n,v_h\rangle_{\partial K}\right),\label{WF:convection}\\
\mathcal{L}_K(\pmb q_h,v_h) &= -\sqrt{\varepsilon}\left((\pmb q_h,\nabla v_h)_K-\langle \hat{\pmb q}_h\cdot \pmb n,v_h\rangle_{\partial K}\right),\label{WF:diffuction1}\\
\mathcal{K}_K^i(u_h,p_h^i) &= -\sqrt{\varepsilon}\left((u_h,(p_h^i)_{x_i})_K-\langle\hat{u}_h, p_h^in_i\rangle_{\partial K}\right),\label{WF:diffuction2}\\
\mathcal{R}_K(u_h,v_h) &= -(r(u_h),v_h)_K,\label{WF:reaction}\\
\mathcal{G}_K(\pmb x,t,v_h) &= (g(\pmb x,t),v_h)_K,\label{WF:force term}
\end{align}
\end{subequations}
where $\hat{f}(u_h)$, $\hat{\pmb q}_h$, and $\hat{u}_h$ are the boundary terms, which are called numerical flux. To guarantee stability, they need specific designs with the two-sided traces on the boundary. Following the choice strategy described in \cite{cockburn1998local,xu2007error}, we take the monotone flux for the convection part and the alternating flux for the diffusion part, defined as
\begin{equation*}
\hat{f}(u_h) = \hat{f}(u_h^-,u_h^+),\ \hat{\pmb q}_h = \pmb q_h^+,\ \hat{u}_h = u_h^-.
\end{equation*}
Note that there are several well-known examples of monotone fluxes $\hat{f}(u_h^-,u_h^+)$, such as the Lax-Friedrichs flux, Godunov flux, and Engquist-Osher flux, etc \cite{cockburn1989tvb}. For the convenience of error analysis, the numerical flux can be written in the following viscosity form
\begin{equation*}
\hat{f}(u_h^-,u_h^+) = \frac{1}{2}\left(f(u_h^-)+f(u_h^+)-\alpha(u_h^-,u_h^+)(u_h^+-u_h^-)\right),
\end{equation*}
with the assumption $\alpha(u_h^-,u_h^+)\ge\alpha_0$ where $\alpha_0$ is a positive number.

Here we denote the global inner product by
\begin{equation*}
(u,v) = \sum_{K\in \mathcal{T}_h}(u,v)_K,\ \langle u,v\rangle = \sum_{e\in \mathcal{E}_h}\langle u,v\rangle_{e}.
\end{equation*}
By summing up the equation \eqref{weak formulation(WF)} over $K\in \mathcal{T}_h$, we obtain the global weak formulation, as follows,
\begin{equation}\label{global weak formulation(gWF)}
	\left\{
	\begin{aligned}
		&((u_h)_t,v_h)=\mathcal{H}(u_h,v_h)+\mathcal{L}(\pmb q_h,v_h)+\mathcal{R}(u_h,v_h)+\mathcal{G}(\pmb x,t,v_h),\ \forall v_h\in U_h^k,\\
		&(q_h^i,p_h^i) = \mathcal{K}^i(u_h,p_h^i),\ \forall p_h^i\in U_h^k,\ i = 1,\dots, d,
	\end{aligned}
	\right.
\end{equation}
where $\mathcal{H},\ \mathcal{L},\ \mathcal{R},\ \mathcal{G},\ \mathcal{K}^i$ are the global operators defined by summing up the corresponding local operators of formulation \eqref{CDR operators} over all elements $K\in \mathcal{T}_h$.
\subsection{Error estimate for the semi-discrete LDG scheme}\label{Error estimate for the RLDG spatial discretization}
The properties in this subsection are the direct extension results in \cite{xu2007error} where we refer the readers for more details. Given the mesh $\mathcal{T}_h$, we first introduce the broken energy space
\[W^{s,p}(\mathcal{T}_h) = \{v\in L^2(\Omega):v|_{K}\in W^{s,p}(K),\ \forall K\in\mathcal{T}_h\},
\]
and denote the corresponding seminorm and norm by
\[|v|_{s,p} = \left(\sum_{K\in\mathcal{T}_h}|v|^p_{s,p,K}\right)^{\frac{1}{p}}
,\ \Vert v\Vert_{s,p} = \left(\sum_{j = 1}^{N}| v|^p_{s,p}\right)^{\frac{1}{p}}.\]
For p = 2, we omit the index $p$ and simply denote these norms by $|v|_s$, and $\Vert v\Vert_s$. Similarly, we can denote the broken energy space on the boundary set by $W^{s,p}(\mathcal{E}_h)$ and its norms by $|v|_{s,p,\mathcal{E}_h}$ and $\Vert v\Vert_{s,p,\mathcal{E}_h}$.

With the above global notation, we present the approximation properties for the global reconstruction operator $R^k$ as follows,
\begin{theorem}
Let $u\in H^{k+1}(\Omega)$, the global approximation $R^ku$ satisfies the following error estimates
	\begin{equation}\label{l2 error estimate for global reconstruction operator}
		\Vert u-R^ku\Vert_{0} \le Ch^{k+1},
	\end{equation}
	\begin{equation}\label{sime norm error estimate for global reconstruction operator}
		|u-R^ku|_{1} \le Ch^k.
	\end{equation}
    \begin{equation}\label{boundary norm and infty norm estimate for global reconstruction operator}
		\Vert u-R^ku\Vert_{0}+h^{\frac{d}{2}}\Vert u-R^ku\Vert_{0,\infty}+ h^{\frac{1}{2}}\Vert u-R^ku\Vert_{0,\mathcal{E}_h}\le Ch^{k+1}.
	\end{equation}
\end{theorem}
\begin{proof}
By summing up all elements, properties \eqref{l2 error estimate for global reconstruction operator} and \eqref{sime norm error estimate for global reconstruction operator} can be derived by the Theorem \ref{Theorem 1} directly. Together with Theorem \ref{Theorem 1} and the Agmon inequality \eqref{Agmon inequality}, the property \eqref{boundary norm and infty norm estimate for global reconstruction operator} can be derived.
\end{proof}

Now, we introduce the error estimate for the semi-discrete LDG scheme. The detailed proof is provided in Appendix \ref{Proof of Theorem3}.
\begin{theorem}\label{l2 estimate for the LDG}
Let $u\in H^{k+1}(\Omega)$ be the exact solution of equation \eqref{CDR equation with initial solution}, and $u_h\in U_h^k$ be the numerical solution of scheme \eqref{global weak formulation(gWF)}. Then for small enough $h$, there exists $C$ such that
	\begin{equation}\label{l2 error estimate for CDR equation}
		\Vert u-u_h\Vert_{0} \le Ch^{k},
	\end{equation}
where $C$ depends on the final time $T$, $k$, $\Vert u\Vert_{H^{k+1}}$, $\Vert \pmb b\Vert_{1,\infty}$, $|\varepsilon|$ and the upper bound of $|f^{(m)}|$ for $m = 1,2,3$.
\end{theorem}
\begin{remark}
The proof of Theorem \ref{l2 estimate for the LDG} follows \cite{xu2007error} where the author presented the results for the one- and two-dimensional cases and the $(k+\frac{1}{2})$th order $L^2$ norm estimate was given. Here, we can only obtain the $k$th order $L^2$ norm estimate in theory, although the $(k+1)$th order convergence results can be observed numerically in Section \ref{Convergence order study}. The main problem is that there are not enough DoFs for us to define some kind of $L^2$-projection with multiple orthogonal properties which always plays an important role in eliminating the order-reduction terms during error estimates, such as the jump terms and the partial terms \cite{ciarlet2002finite}.
\end{remark}

\subsection{Fully discrete IMEX RK LDG scheme}
In this section, we would like to introduce the third-order IMEX RK method \cite{calvo2001linearly} for time discretization, which treats the nonlinear terms, such as convection, reaction, and source terms explicitly but the linear diffusion term implicitly. It is a kind of balanced scheme which does not subject to severe time step restriction while avoiding large nonlinear system solver. The simplicity and good performance of this method make it famous among many others, although the choice of implicit parts is too rough to handle the case of very stiff reaction terms. For the IMEX RK method coupled with the semi-discrete LDG scheme for solving convection-diffusion problems, Wang et al. have proposed a much weaker stability condition $\Delta t\le C$ with a constant $C$, where $\Delta t$ is the time step \cite{wang2015stability}. However, here we only consider a standard CFL stability condition $\Delta t\le Ch$ for our nonlinear CDR equation.

To describe the IMEX RK method distinctly, we consider the system \eqref{global weak formulation(gWF)} as the following simple ordinary differential equation (ODE) form
\begin{equation*}
 \begin{aligned}
\frac{d U(t)}{d t} &= L^1(t,\pmb Q)+N(t,U),\\
\pmb Q &= L^2(t,U),\\
U(0) &= U^0,
\end{aligned}
\end{equation*}
where $U$ is the undetermined coefficient vector, $L^1(t,\pmb Q)$ and $L^2(t,U)$ represent the linear diffusion part, $N(t,U)$ represents the nonlinear part, such as convection, reaction, and force part, in semi-discrete LDG scheme. Let $U^n$ be the numerical approximation at time $t^n$. One step time evolution of the three-stage third-order IMEX RK scheme is given by
\begin{equation}\label{three-order IMEX RK scheme}
 \begin{aligned}
U^{(1)} &= U^n,\\
U^{(i)} &= U^n+\Delta t\sum_{j = 2}^ia_{ij}L^1(t_j^n,\pmb Q^{(j)})+\Delta t\sum_{j = 1}^{i-1}\hat{a}_{ij}N(t_j^n,U^{(j)}),\ 2\le i\le 4,\\
U^{n+1} &= U^n+\Delta t\sum_{i = 2}^{4}b_{i}L^1(t_i^n,\pmb Q^{(i)})+\Delta t\sum_{i = 1}^{4}\hat{b}_{i}N(t_i^n,U^{(i)}),\\
\pmb Q^{(i)} &= L^2(t_i^n,U^{(i)}),\ 2\le i\le 4.\\
\end{aligned}
\end{equation}
Where $U^{(i)}$ denotes the intermediate stages at the corresponding time $t_i^n = t^n+c_i\Delta t$. All the coefficients in the scheme are presented in the following Butcher tableau
\begin{table}[H]
\centering
\renewcommand{\arraystretch}{1.5}
\begin{tabular}{c|c|c}
$\pmb c$& $A$&$\hat{A}$\\
\hline
& $\pmb b^T$&$\hat{\pmb b}^T$\\
\end{tabular}
\end{table}
\noindent which is specified as follows,
\begin{table}[H]
\centering
\renewcommand{\arraystretch}{1.5}
\begin{tabular}{c|c c c c|c c c c}
0&0&0&0&0&0&0&0&0\\
$\gamma$&0&$\gamma$&0&0&$\gamma$&0&0&0\\
$\frac{1+\gamma}{2}$&0&$\frac{1-\gamma}{2}$&$\gamma$&0&$\frac{1+\gamma}{2}-\alpha_1$&$\alpha$&0&0\\
1&0&$\beta_1$&$\beta_2$&$\gamma$&0&1-$\alpha_2$&$\alpha_2$&0\\
\hline
& 0&$\beta_1$&$\beta_2$&$\gamma$&0&$\beta_1$&$\beta_2$&$\gamma$\\
\end{tabular}
\end{table}
\noindent where the parameters are selected as $\gamma\approx  0.435866521508459$, $\beta_1 = -\frac{3}{2}\gamma^2+4\gamma-\frac{1}{4}$, $\beta_2 = \frac{3}{2}\gamma^2-5\gamma-\frac{5}{4}$, $\alpha_1 = -0.35$, and $\alpha_2 = \frac{\frac{1}{3}-2\gamma^2-2\beta_2\alpha_1\gamma}{\gamma(1-\gamma)}$ \cite{calvo2001linearly}.

So far, we have developed the fully discrete IMEX RK LDG scheme by combining the semi-discrete LDG scheme \eqref{global weak formulation(gWF)} and the IMEX RK scheme \eqref{three-order IMEX RK scheme}. In this type of scheme, repetitively solving large linear systems is the main challenge for the computation cost. This allows the full benefit of RDG space to be realized in terms of cost savings.


\section{Numerical examples}\label{Sec:Numerical examples}
In this section, we would like to validate our scheme and study its numerical behavior. The numerical results focus mainly on the following aspects: First, we demonstrate the convergence behavior of the third-and sixth-order IMEX RK LDG scheme proposed based on the RDG space. Secondly, we investigate the ability to capture sharp layers for our method, which demonstrates the local property of our reconstruction operator.
\subsection{Convergence order study}\label{Convergence order study}

\paragraph{1D-Test 1}Consider the linear convection-diffusion equation
\[
\begin{aligned}
&u_t+u_x-u_{xx}=g(t,x),\ x\in[0,2\pi].
\end{aligned}
\]
\paragraph{1D-Test 2}Consider the nonlinear convection-diffusion equation
\[
\begin{aligned}
&u_t+(\frac{u^2}{2})_x-u_{xx} =g(t,x),\ x\in[-\pi,\pi].
\end{aligned}
\]
The same initial condition $u(x,0) = sin(x)$ and periodic boundary condition are used for the above two tests. Moreover, the same exact solution is given by $u(x,t) = \sin(x-t)$. Then the right-hand sides of these two tests can be obtained by a simple calculation.
\paragraph{1D-Test 3}We seek traveling wave solutions for the equation
\[
\begin{aligned}
&u_t-\nu^2 u_{xx}+u^3-u=0,\ x\in[-1,1],\\
&u(x,0) = \frac{1}{2}(1-\tanh(\frac{x}{2\sqrt{2}\nu})).\\
\end{aligned}
\]
The Dirichlet boundary condition is given by the exact solution $u(x,t) = \frac{1}{2}(1-\tanh(\frac{x-st}{2\sqrt{2}\nu}))$, where $s = 3\nu/\sqrt{2}$ is the speed of the traveling wave. Here, the value of the diffusion velocity is $\nu^2 = 0.01$.
\paragraph{2D-Test 1}Consider linear convection diffusion equation
\[
\begin{aligned}
&u_t+\nabla\cdot u-\Delta u=g(t,x,y),\ x,y\in[0,2\pi]\times[0,2\pi].
\end{aligned}
\]
\paragraph{2D-Test 2}Consider nonlinear convection diffusion equation
\[
\begin{aligned}
&u_t+\nabla\cdot (\frac{u^2}{2})-\Delta u =g(t,x,y),\ x,y\in[0,2\pi]\times[0,2\pi].
\end{aligned}
\]
For the 2D-Test1 and 2D-Test2, let $u(x,y,t) = \sin(x+y-2t)$ be the exact solution. Naturally, the initial value conditions and the right-hand sides of these two tests are given. Here we consider the periodic boundary condition.
\paragraph{2D-Test 3}Consider the Allen-Cahn equation
\[
u_t-\Delta u + \frac{1}{\nu^2}(u^3-u)=g(t,x,y),\ x,y\in[0,2\pi]\times[0,2\pi],
\]
with the initial value $u(x,y,0) = sin(x+y)$ and the periodic boundary condition. Here, $\nu = 0.3$. The source term $g(x,y,t)$ is given by the exact solution $u(x,y,t) = \exp(-2t)\sin(x+y)$.

With the terminal time $t = 1$ and the fixed CFL condition $\Delta t = h$, we apply the third- and sixth-order IMEX RK LDG scheme with the RDG space (i.e. $k = 2, 5$) for solving 1D-Test1 to 3 on a series of meshes with the number of the elements $N = 2^4,2^5,2^6,\dots,2^9$. The $L^2$-norm errors in approximation to the exact solution $u$ and its first derivative $u_x$ are presented in Figure \ref{fig:error_1D}. We can observe that both the error $\Vert u-u_h\Vert_{L^2}$ and the error $\Vert u_x-q_h\Vert_{L^2}$ converge at the optimal rate $O(h^{k+1})$ for these tests. 
For two dimensions, let the CFL number be $0.5$ for 2D-Test 1 and 2D-Test 2, and $0.2$ for 2D-Test 3. Here, we consider the uniform rectangular meshes with the number of the elements $N = 20^2,30^2,40^2,\dots,70^2$ for the computation. We also compute the errors at $t=1$. Figure \ref{fig:error_2D} demonstrates the results which show the same conclusion as the one-dimensional case.
\begin{figure}[htbp]
\centering
\includegraphics[width=0.48\textwidth]{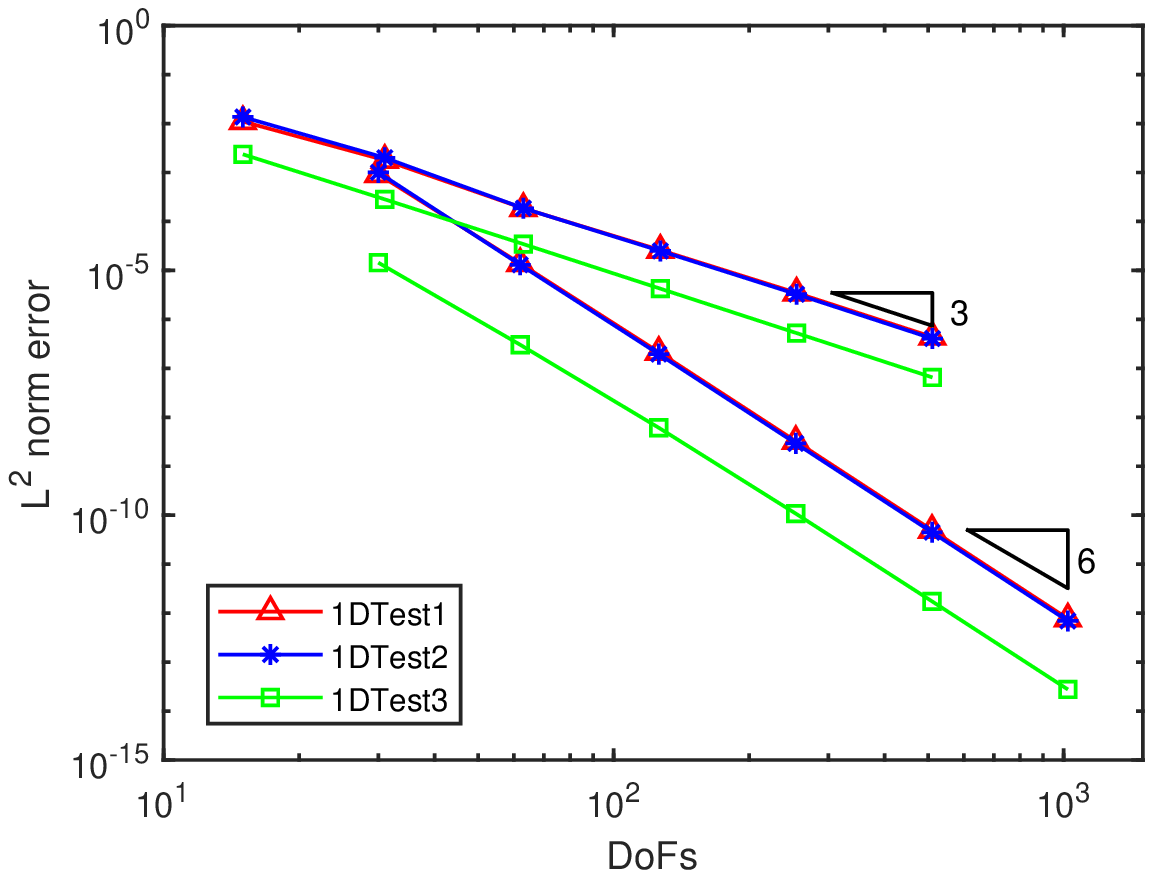}
\includegraphics[width=0.48\textwidth]{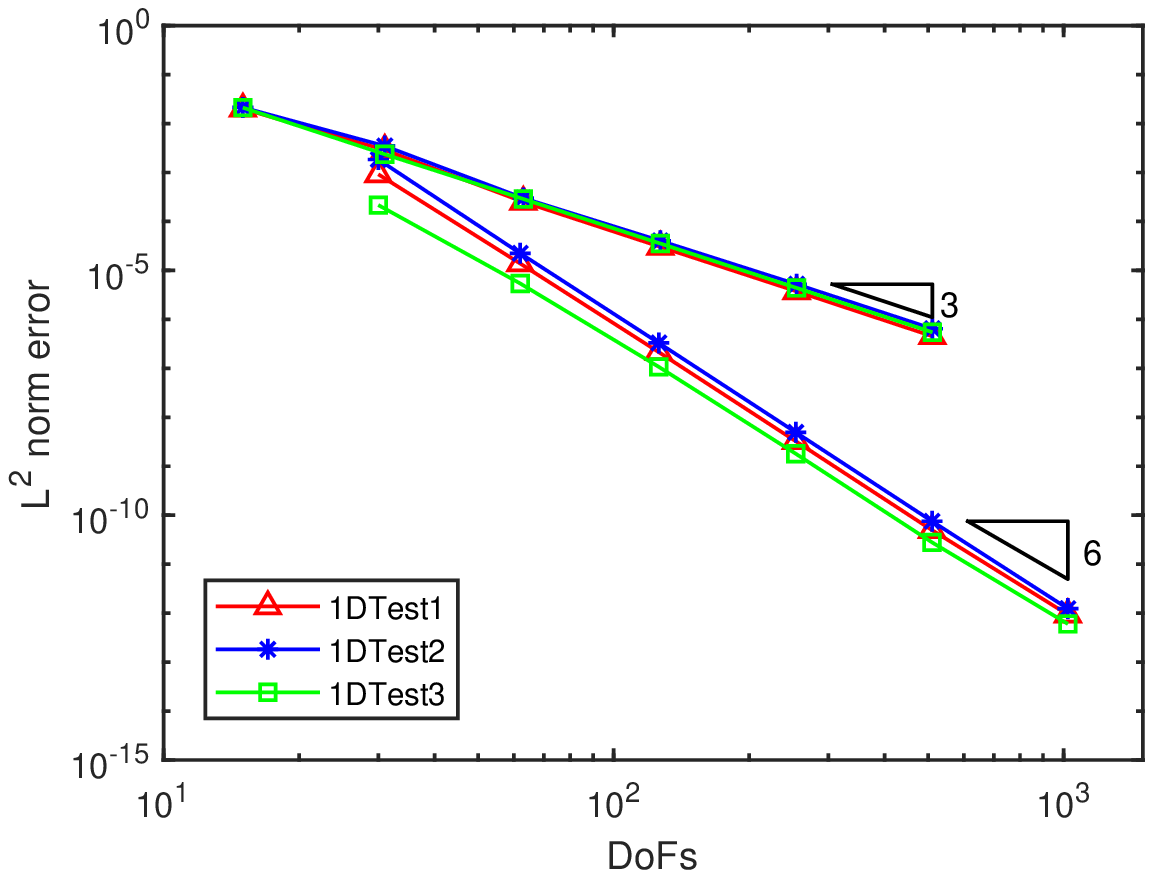}\\
\caption{The convergence rates of $\Vert u-u_h\Vert_{L^2(\mathcal{T}_h)}$ (left) and $\Vert u_x-q_h\Vert_{L^2(\mathcal{T}_h)}$ (right).}
\label{fig:error_1D}
\end{figure}
\begin{figure}[htbp]
\centering
\includegraphics[width=0.48\textwidth]{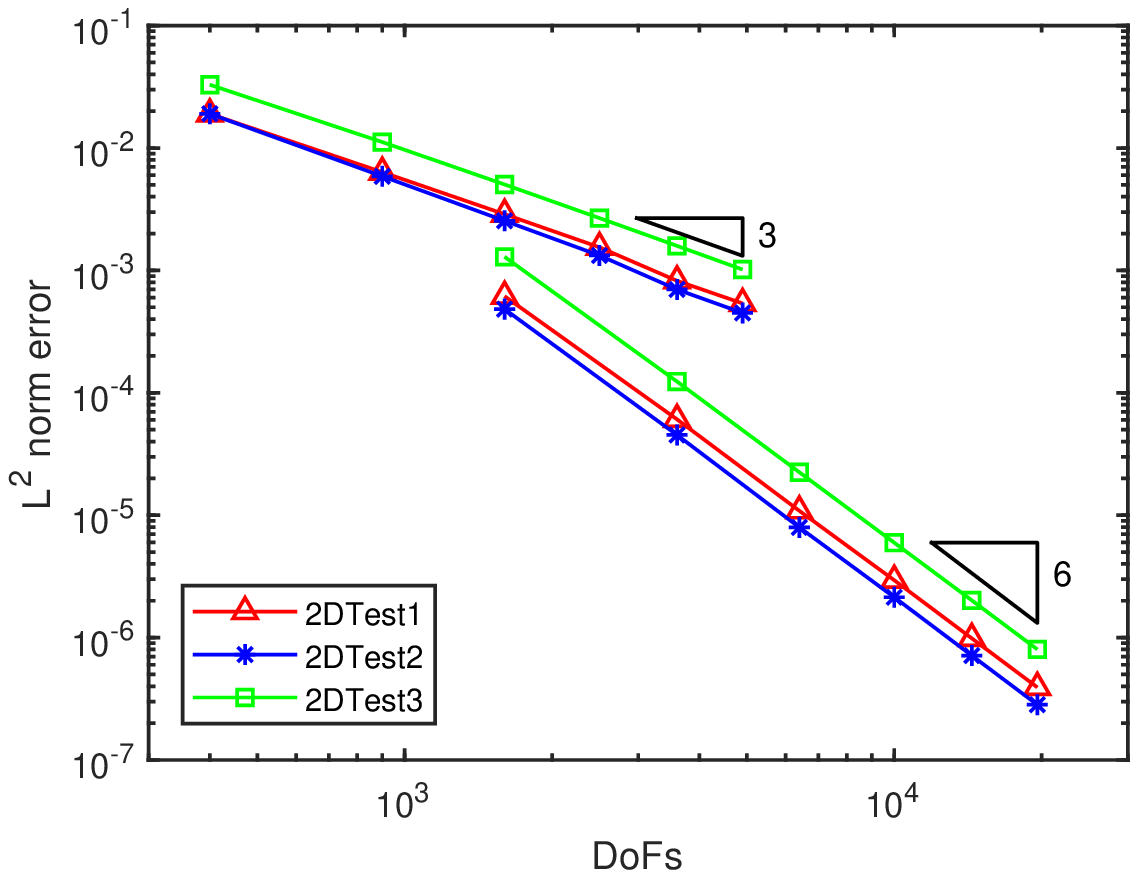}
\includegraphics[width=0.48\textwidth]{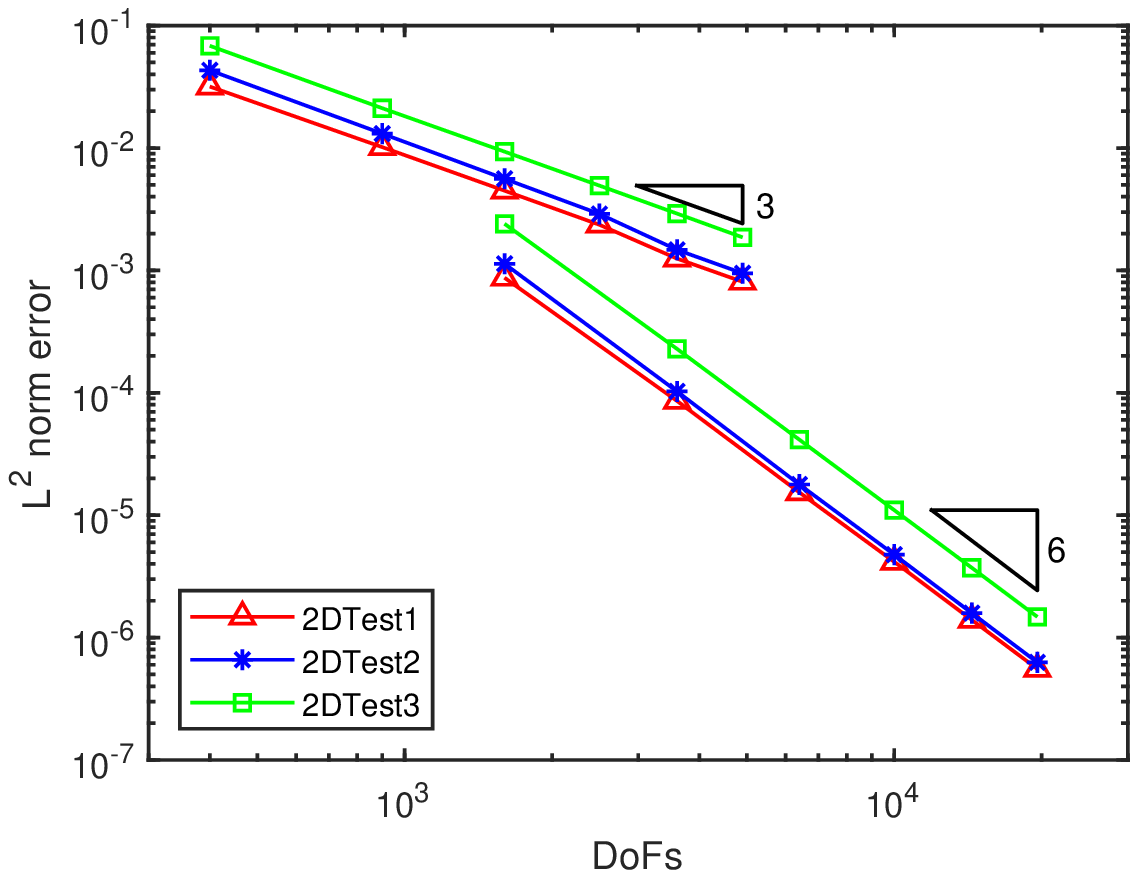}
\caption{The convergence rates of $\Vert u-u_h\Vert_{L^2(\mathcal{T}_h)}$ (left) and $\Vert \nabla u-\pmb q_h\Vert_{L^2(\mathcal{T}_h)}$ (right).}
\label{fig:error_2D}
\end{figure}

\subsection{Local property verification}
\paragraph{1D-Test 4}Consider the viscous Burgers equation with the Dirichlet boundary condition
\[
\begin{aligned}
&u_t+(\frac{u^2}{2})_x-\nu u_{xx}+\pi\cos(\pi x)u=0,\ x\in[0,1],\\
&u(0,t)= 1,\ u(1,t) = -0.1.
\end{aligned}
\]
The initial condition is given as follows,
\[
u(x,0) =
\left\{
\begin{aligned}
&1,\ 0\le x\le 0.3,\\
&-0.1,\ 0.3< x\le 1.
\end{aligned}
\right.
\]
\paragraph{2D-Test 4}Consider the viscous Burgers problem in \cite{nguyen2009implicit}
\[
\begin{aligned}
&u_t+\nabla\cdot (\frac{u^2}{2})-\nu\Delta u =g(x,y,t),\ x,y\in[0,1]\times[0,1],
\end{aligned}
\]
with the homogeneous Dirichlet boundary condition and the initial condition $u(x,y,0)$ $ = 0$. The source term $g(x,y,t)$ is given by the exact solution $u(x,y,t) = (\exp(t)-1)xy\tanh(\frac{1-x}{\nu})\tanh(\frac{1-y}{\nu})$.
\paragraph{2D-Test 5}Consider the rigid body rotation problem
\[
\begin{aligned}
&u_t+\nabla \cdot(\pmb\alpha(x,y)u)-\nu\Delta u = 0,\ x,y\in[-2\pi,2\pi]\times[-2\pi,2\pi],
\end{aligned}
\]
where $\pmb\alpha(x,y) = (-y,x)^T$. Here we consider the same boundary condition as 2D-Test 4. Following \cite[Example 3.4]{ding2020semi}, the initial condition includes a slotted disk, a cone, and a smooth hump.

We take the diffusion viscosity coefficients $\nu = 10^{-3}$ for 1D-Test4 and 2D-Test5, and $\nu = 10^{-2}$ for 2D-Test4.
For 2D-Test 4, we have a representation of the exact solution. For 1D-Test 4 and 2D-Test 5, without the exact solutions, let the numerical solutions solved on the fine meshes be regarded as the surrogate exact solutions.
In Figure \ref{fig:exact solution with layers}, we present these exact solutions at $t = 1$. We can find that there are layers near $x = 0.5$, along $x = 0$ and $y=1$, and on the boundary of a slotted disk for the solutions of 1D-Test 4, 2D-Test4, and 2D-Test 5 respectively. Note that the position of the layer for 2D-Test 5 moves in a counterclockwise direction with time, which is challenging for numerical simulation.
\begin{figure}[htbp]
\centering
\includegraphics[width=0.48\textwidth]{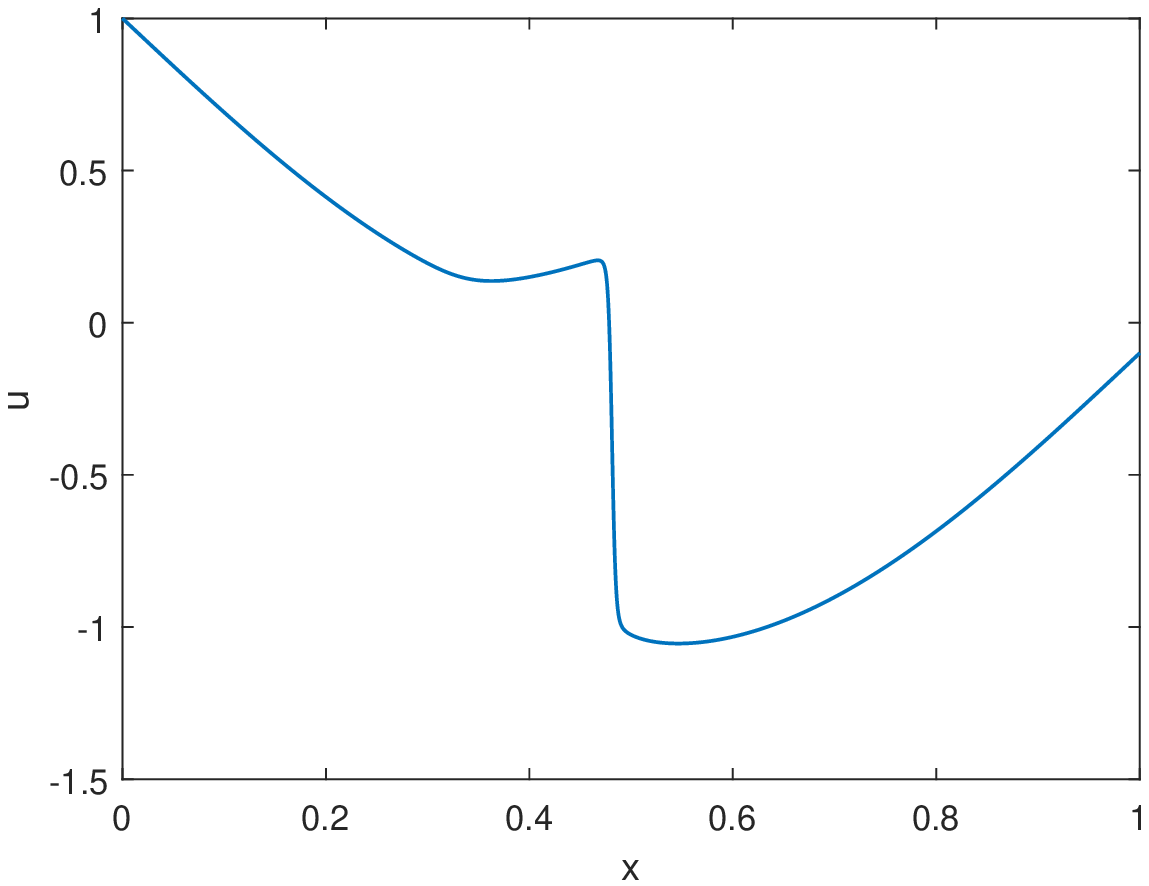}
\includegraphics[width=0.48\textwidth]{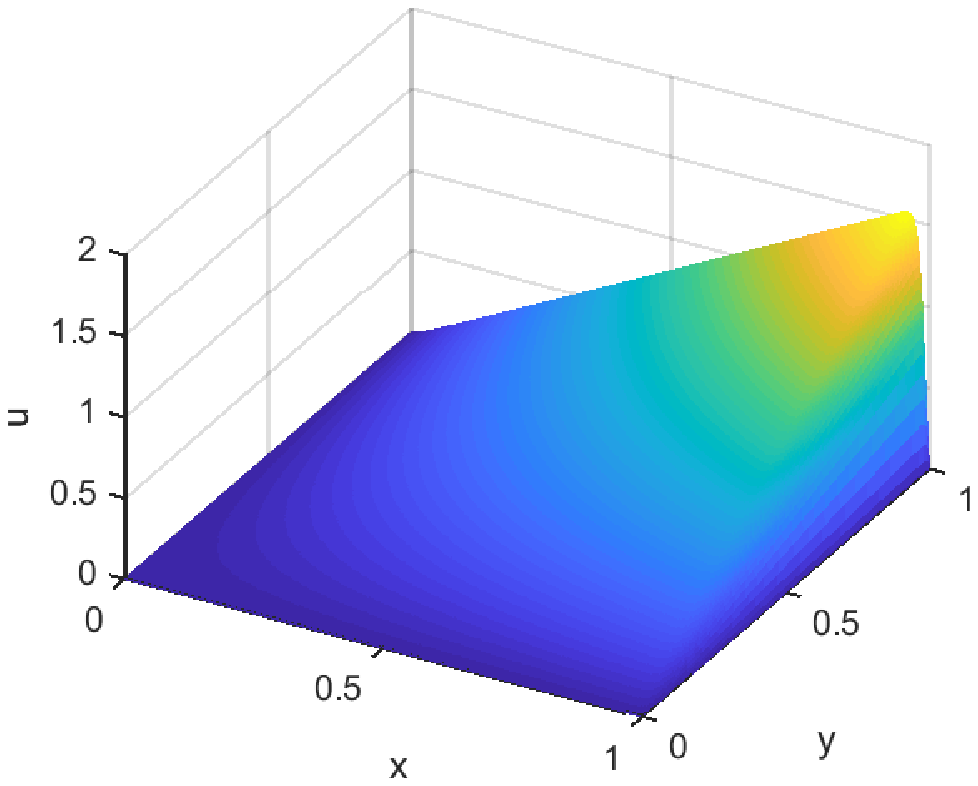}

\includegraphics[width=0.48\textwidth]{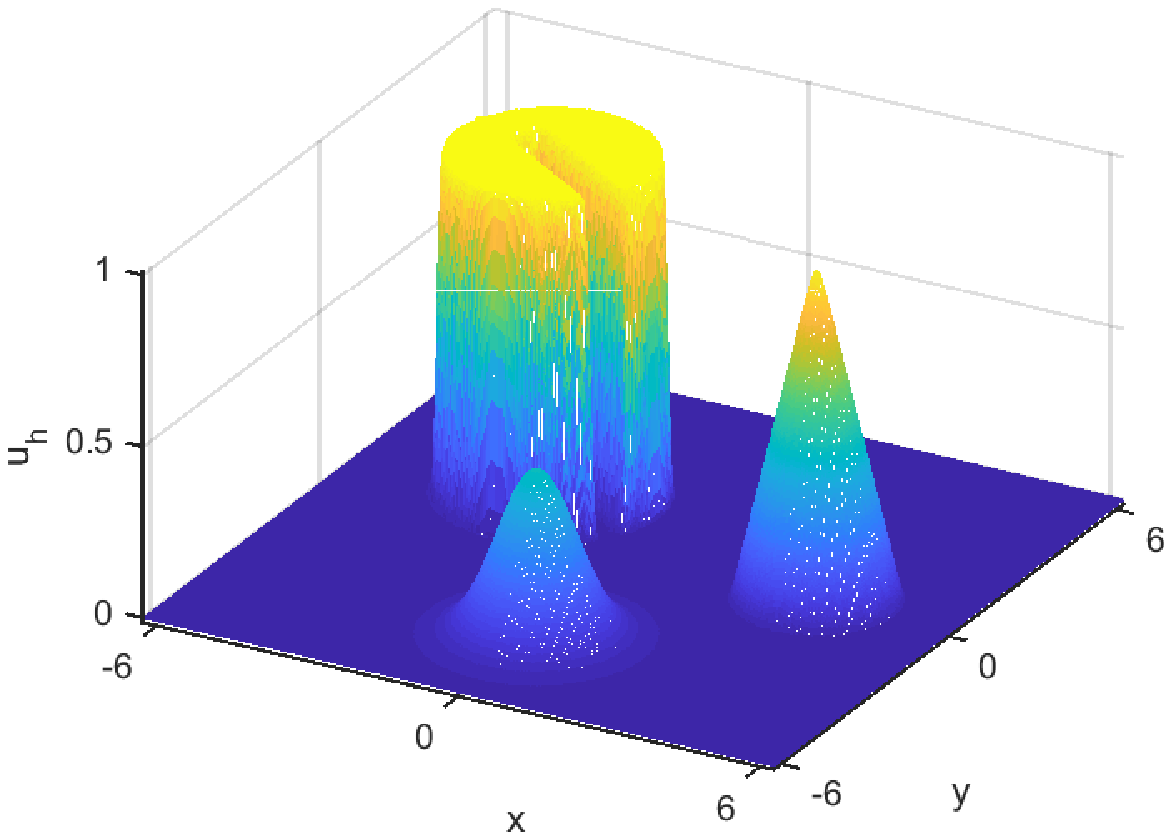}
\includegraphics[width=0.45\textwidth]{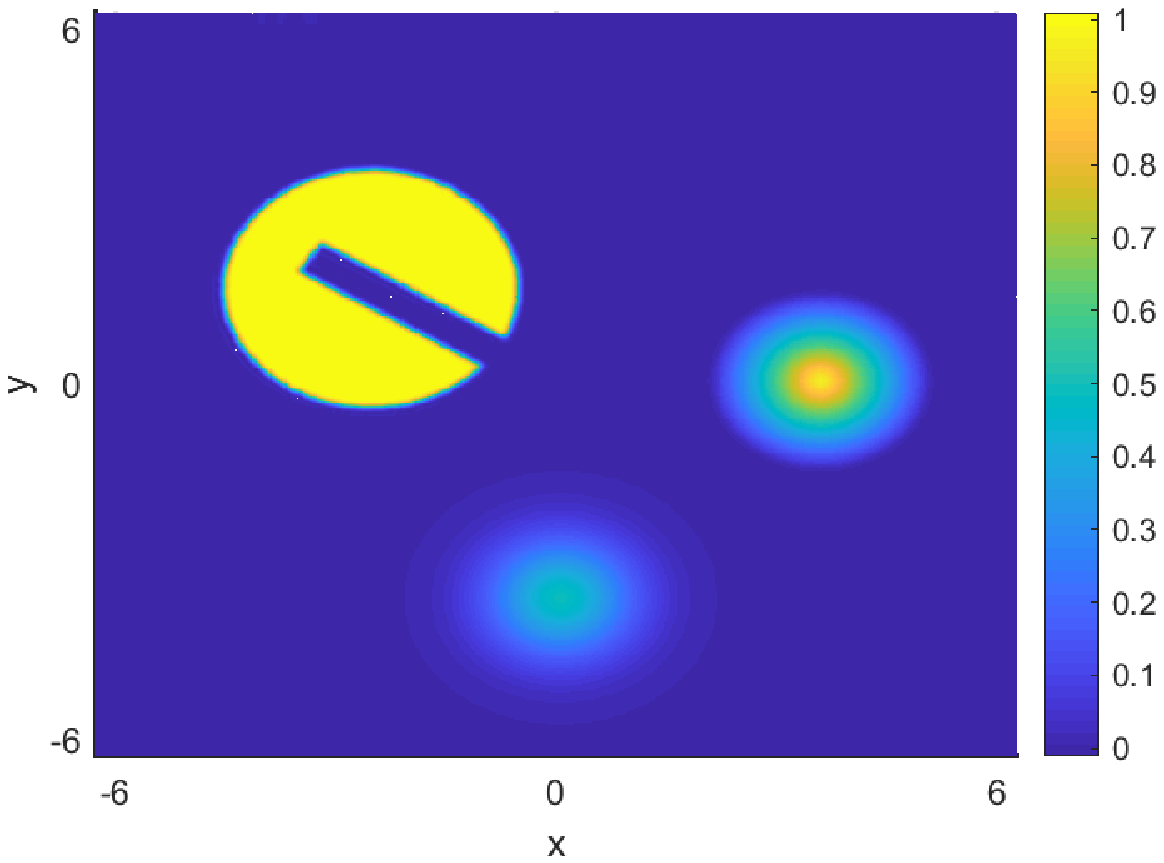}D
\caption{The image of the (surrogate) exact solutions for 1D-Test 4 (Left Top), 2D-Test 4 (Right Top), and 2D-Test 5 (Left Bottom). And the overlooking 2D image for 2D-Test 5 (Right Bottom). }
\label{fig:exact solution with layers}
\end{figure}

We apply the third- and sixth-order methods proposed in this article to solve the above three tests with different numbers of elements $N$. The plots of the numerical solutions
are presented in Figure \ref{fig:numerical solution 1DTest4}-\ref{fig:numerical solution 2DTest5}. We can observe that, when coarse mesh,  there is some overshooting/undershooting near the layer or the region where the layer goes through along with the time. As the mesh is refined, the overshooting/undershooting becomes smaller and smaller until the layers can be well captured.
Under the same mesh size, the higher-order method shows more excellent performance in capturing the layer. This is not only because of its intrinsically higher approximation accuracy but also as the fixed narrow stencil applied in our method guarantees the local data structure property for different order methods. 
\begin{figure}[htbp]
\centering
\includegraphics[width=0.3\textwidth]{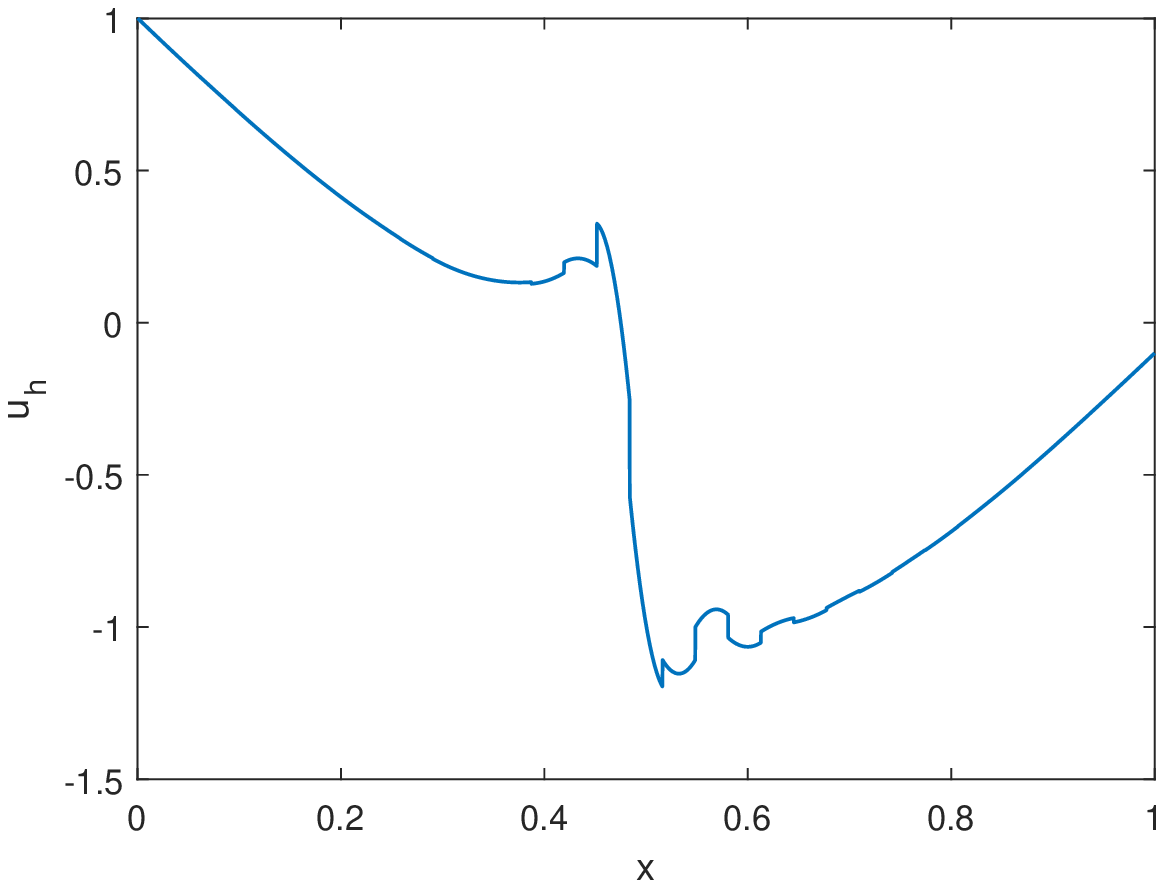}
\includegraphics[width=0.3\textwidth]{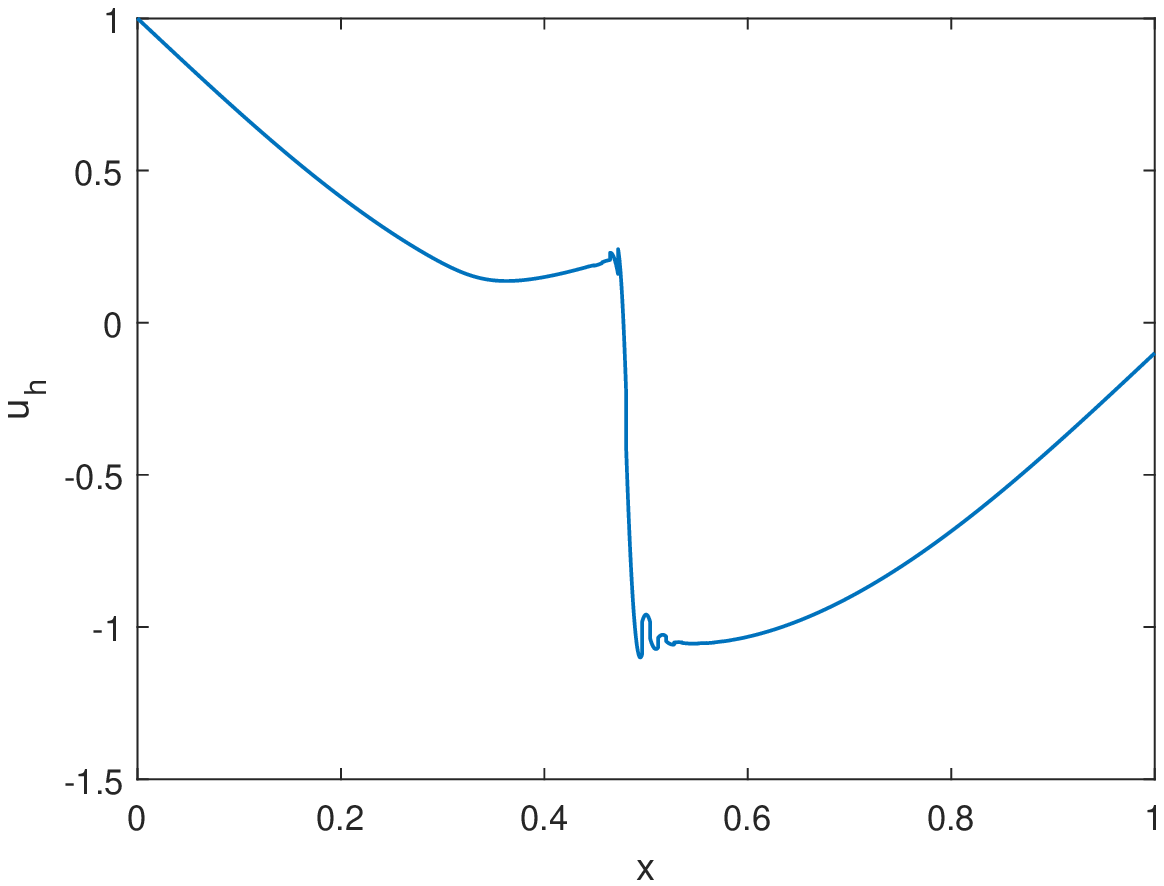}
\includegraphics[width=0.3\textwidth]{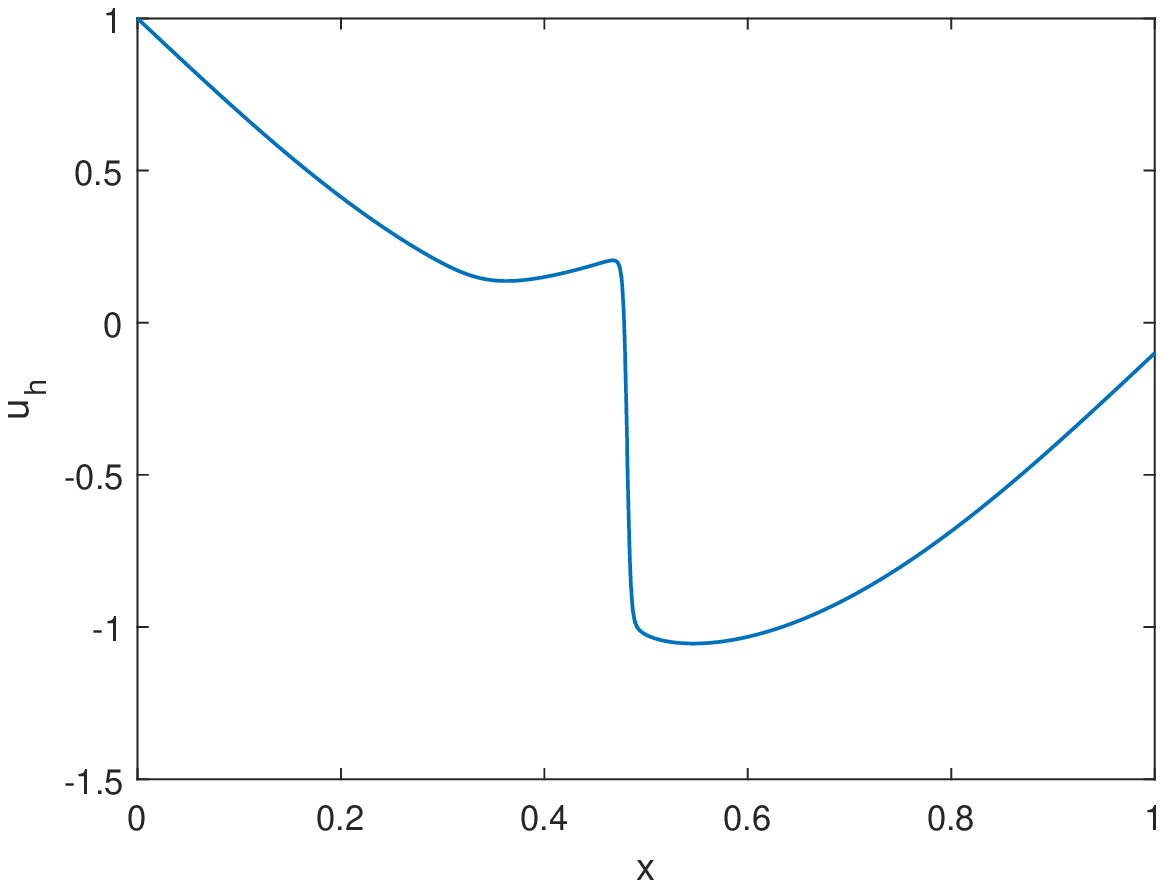}

\includegraphics[width=0.3\textwidth]{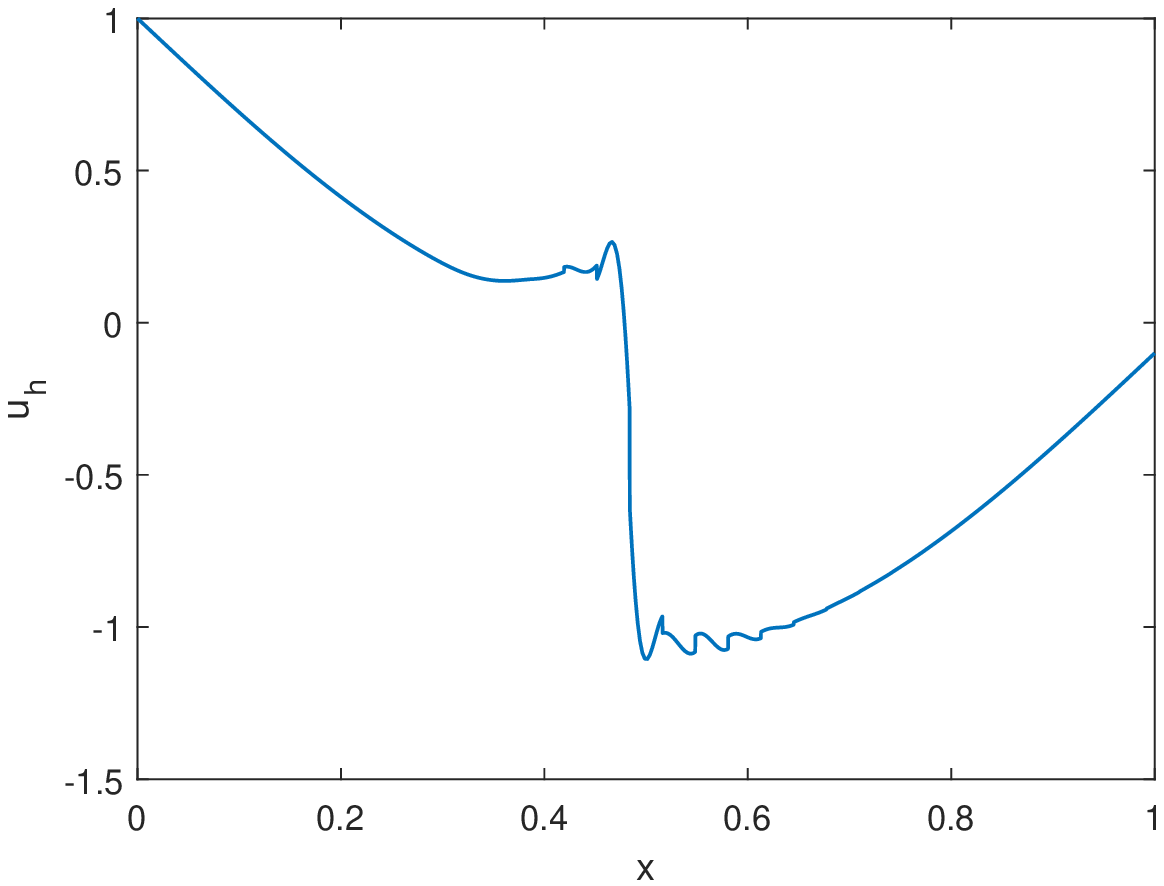}
\includegraphics[width=0.3\textwidth]{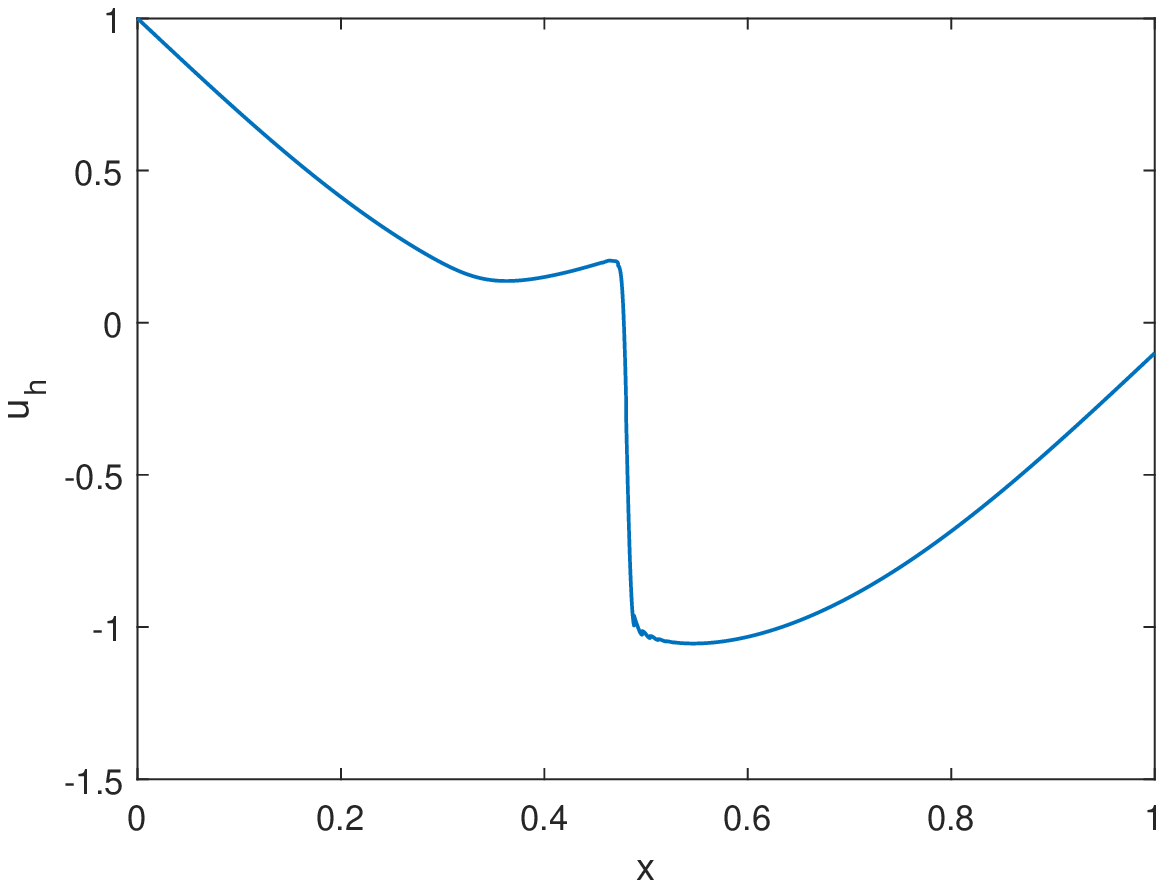}
\includegraphics[width=0.3\textwidth]{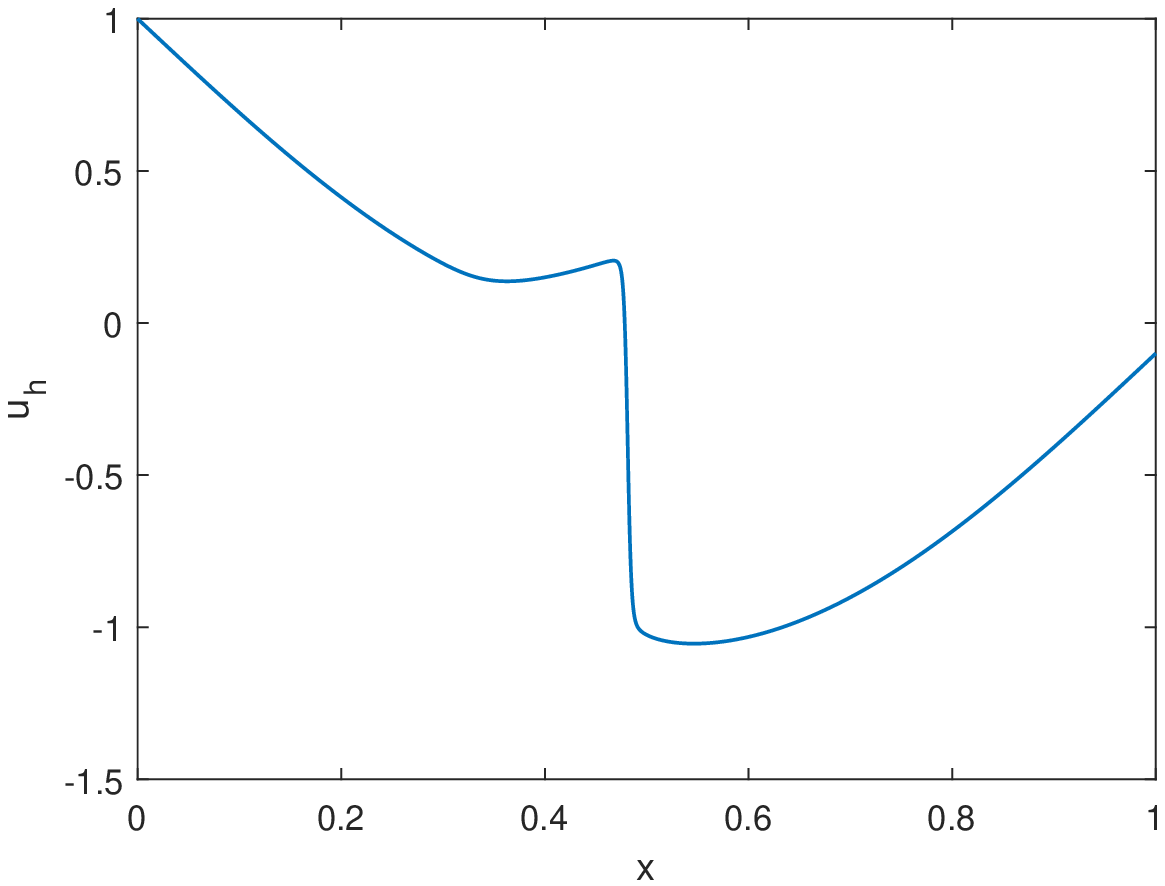}

\caption{The image of the numerical solutions for 1D-Test 4. Top row: numerical solutions for the third-order method with $N = 32, 128, 512$ (from Left to Right). Bottom row: numerical solutions for the sixth-order method.}
\label{fig:numerical solution 1DTest4}
\end{figure}

\begin{figure}[htbp]
\centering
\includegraphics[width=0.3\textwidth]{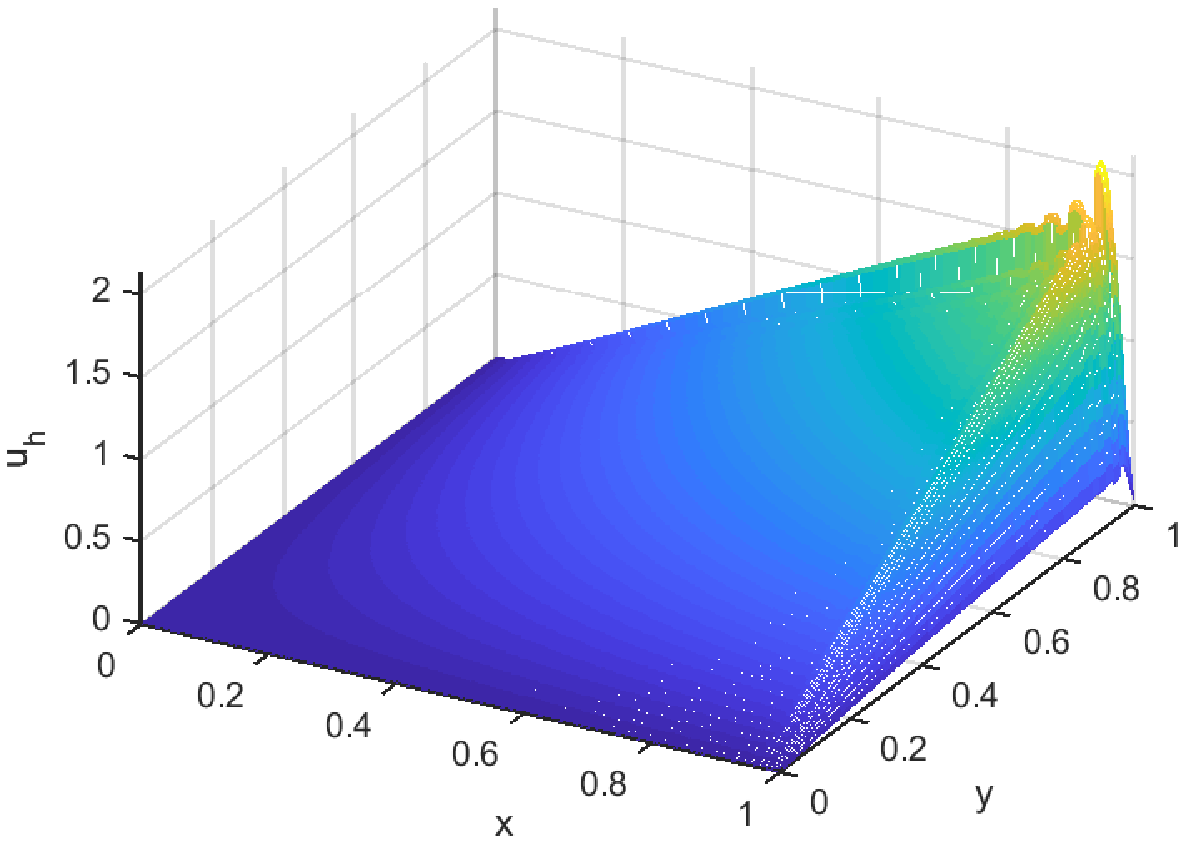}
\includegraphics[width=0.3\textwidth]{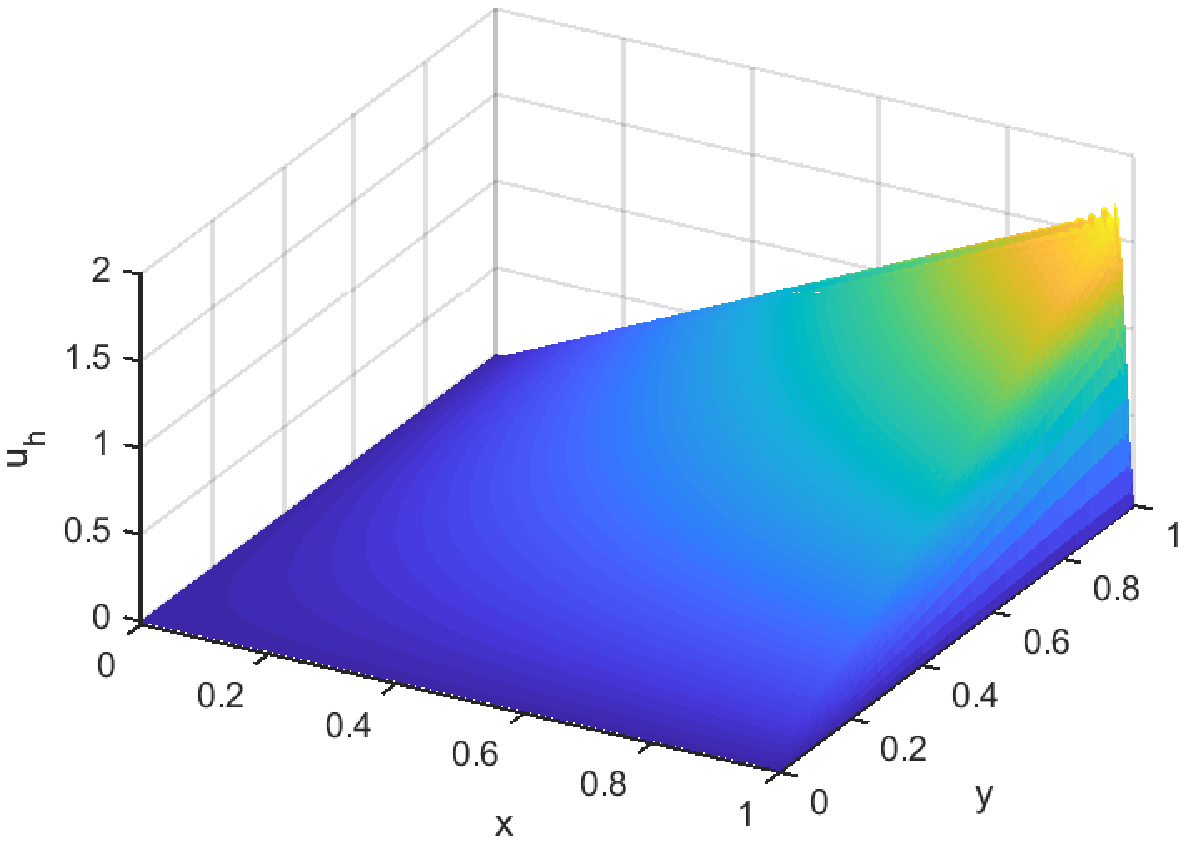}
\includegraphics[width=0.3\textwidth]{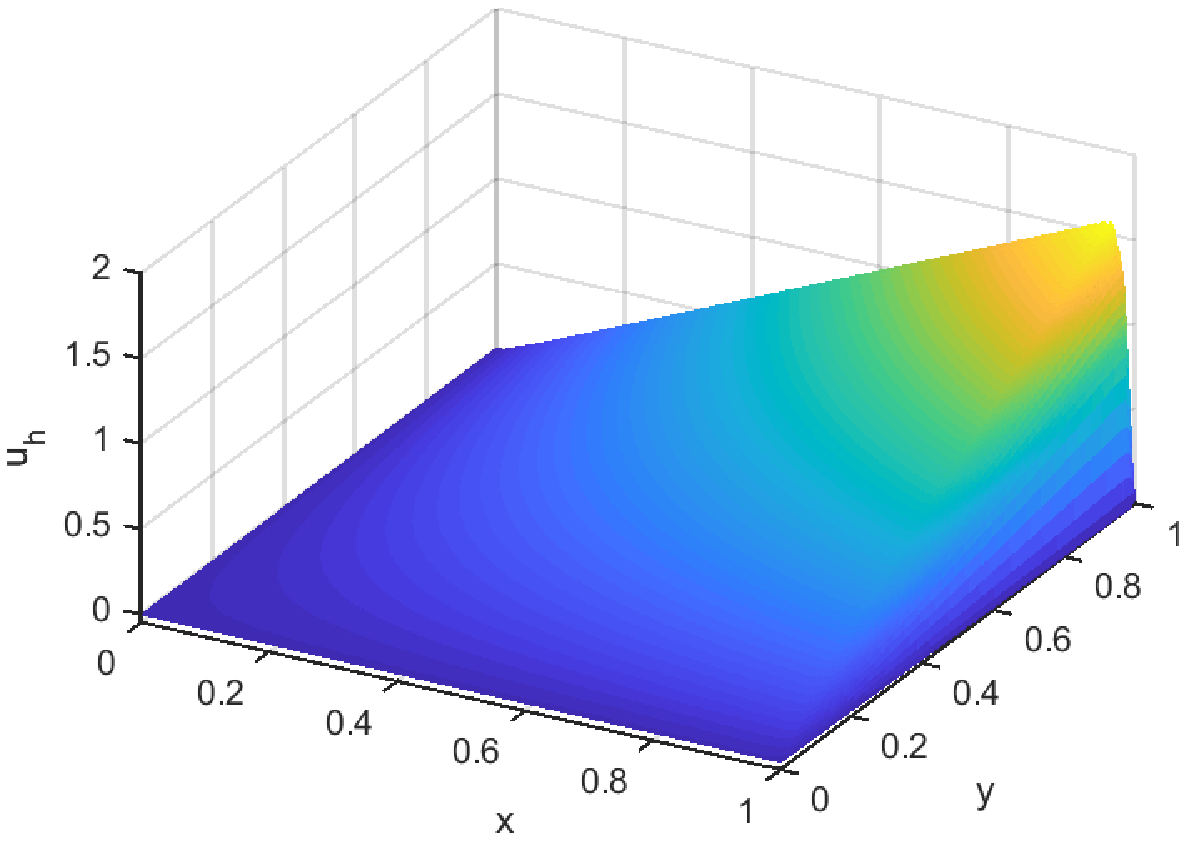}

\includegraphics[width=0.3\textwidth]{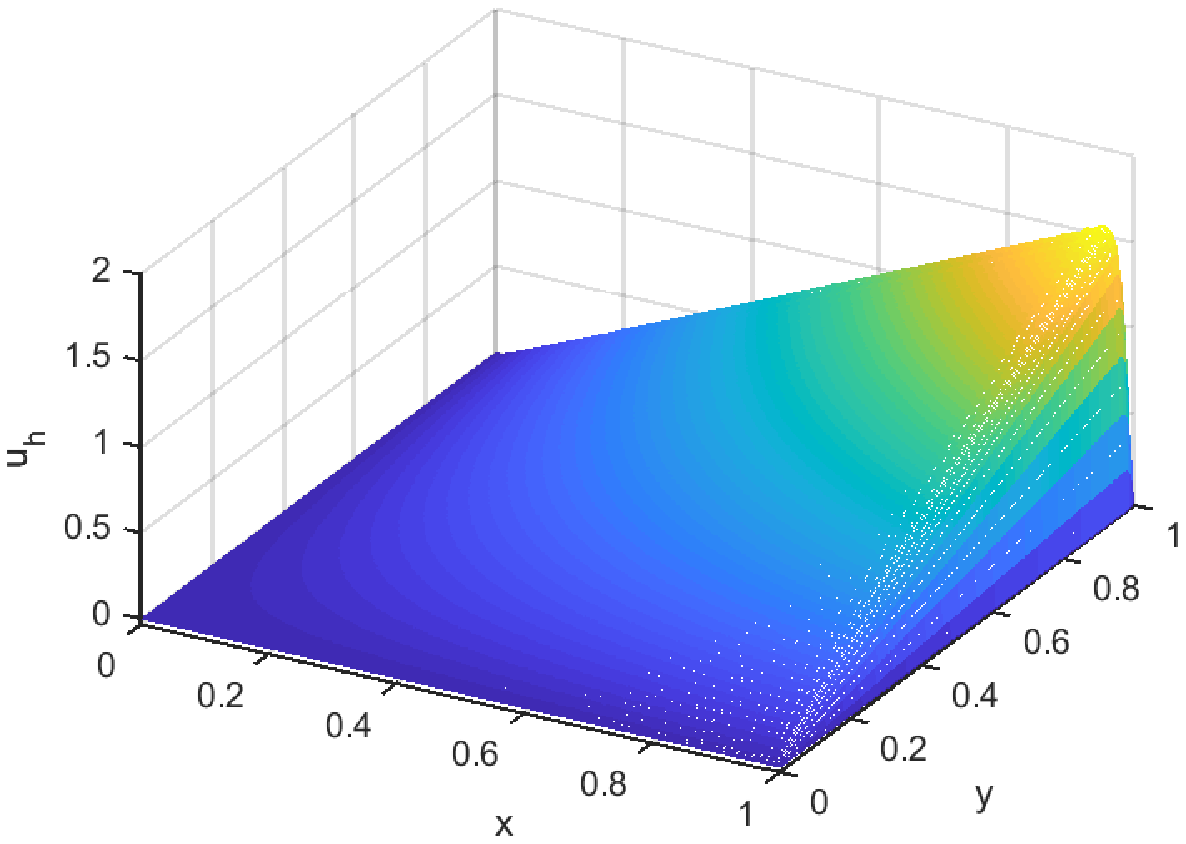}
\includegraphics[width=0.3\textwidth]{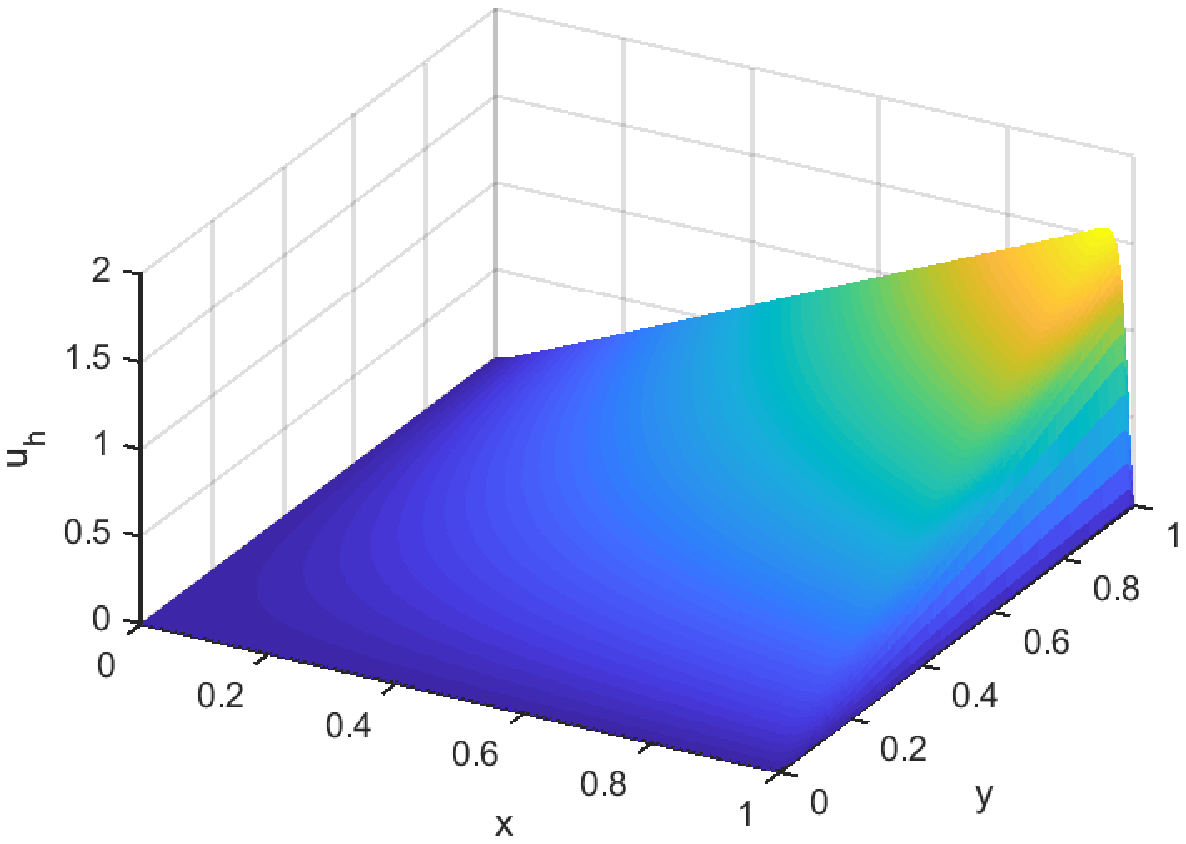}
\includegraphics[width=0.3\textwidth]{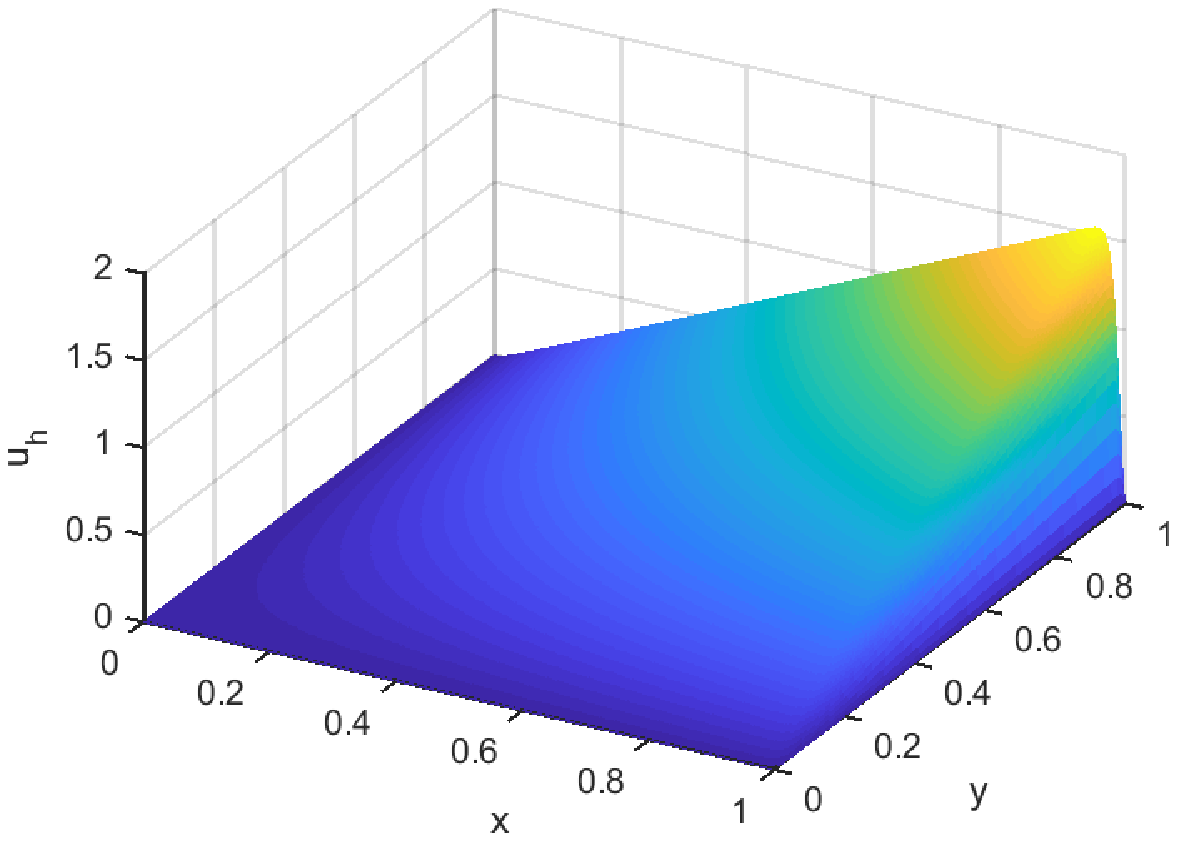}
\caption{The image of the numerical solutions for 2D-Test 4. Top row: numerical solutions for the third-order method with $N = 50^2, 100^2, 150^2$ (from Left to Right). Bottom row: numerical solutions for the sixth-order method.
}
\label{fig:numerical solution 2DTest4}
\end{figure}
\begin{figure}[htbp]
\centering
\includegraphics[width=0.3\textwidth]{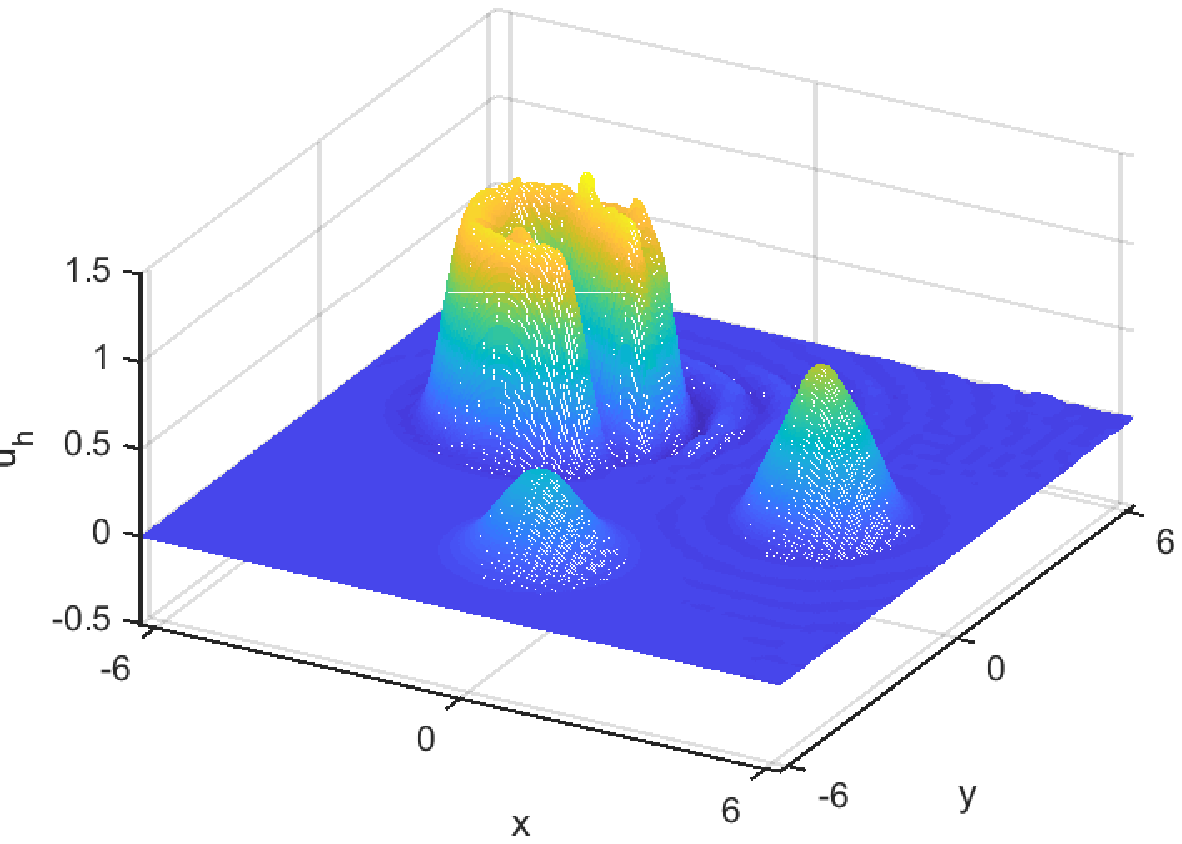}
\includegraphics[width=0.3\textwidth]{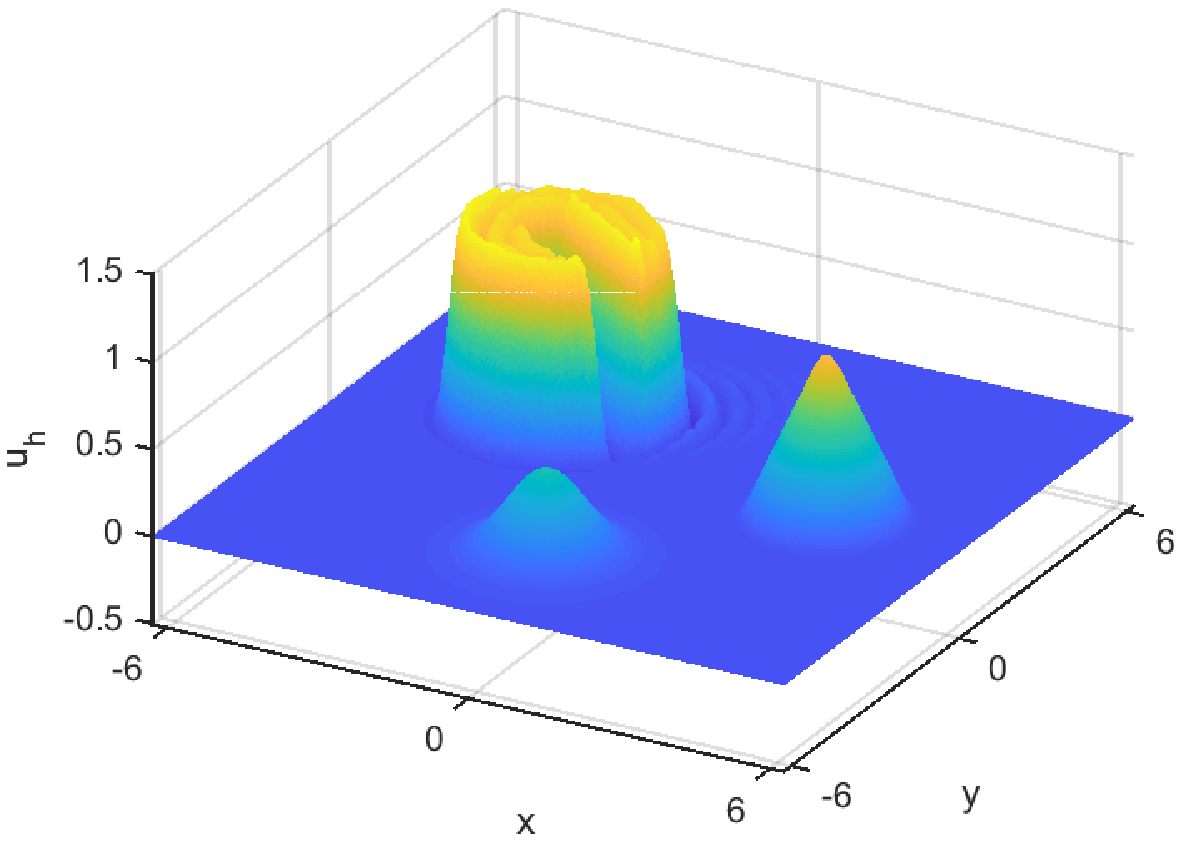}
\includegraphics[width=0.3\textwidth]{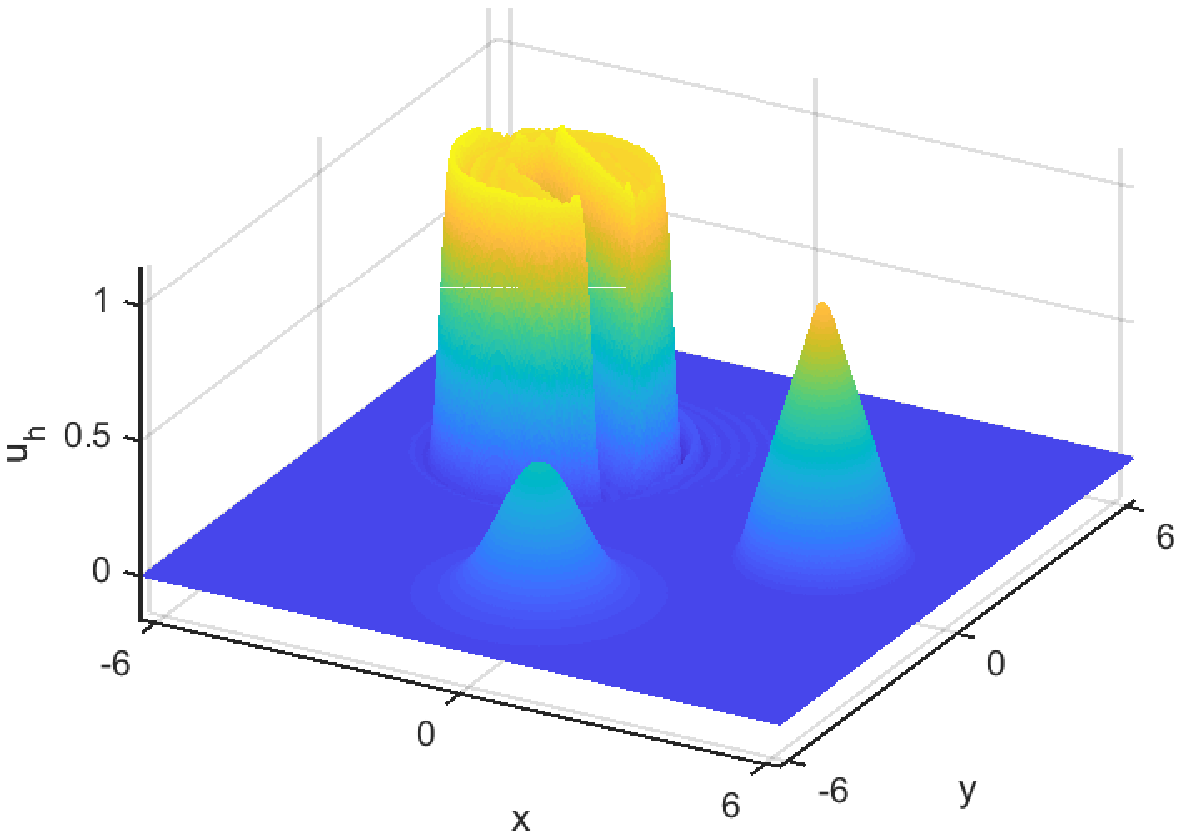}

\includegraphics[width=0.3\textwidth]{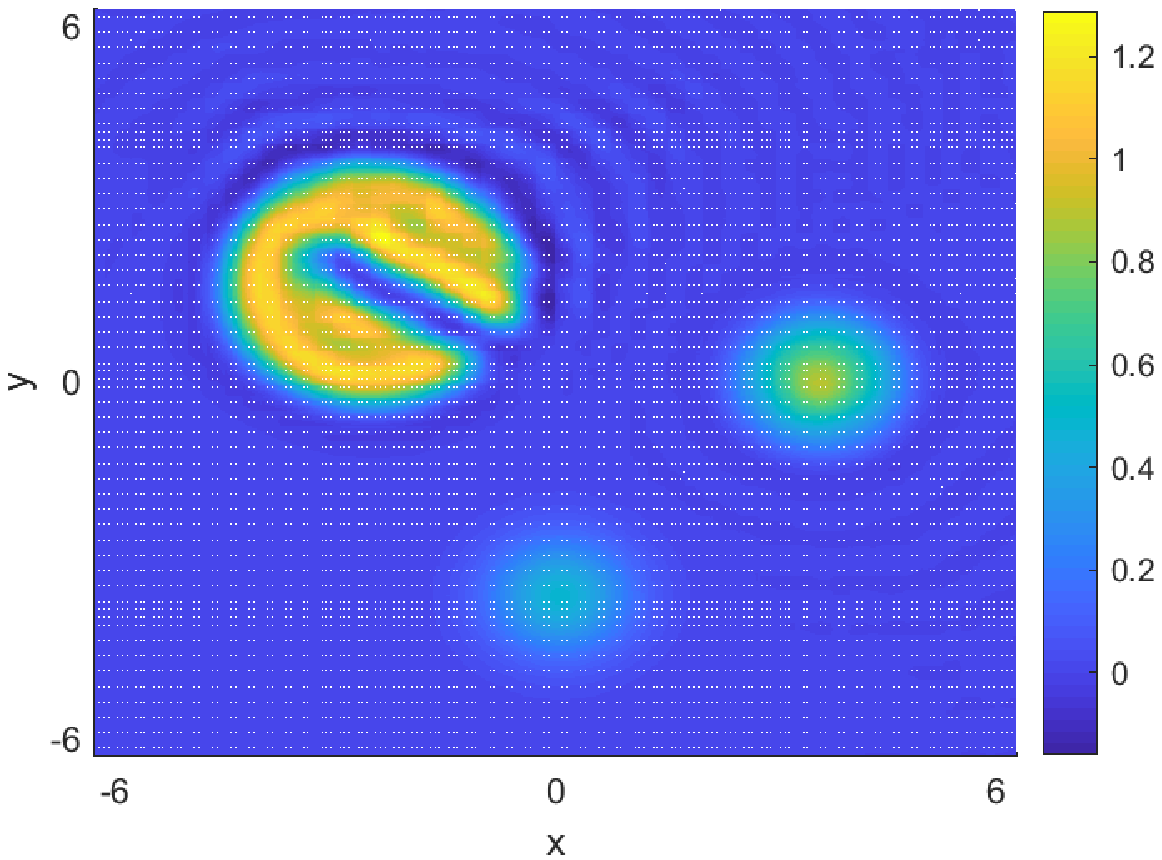}
\includegraphics[width=0.3\textwidth]{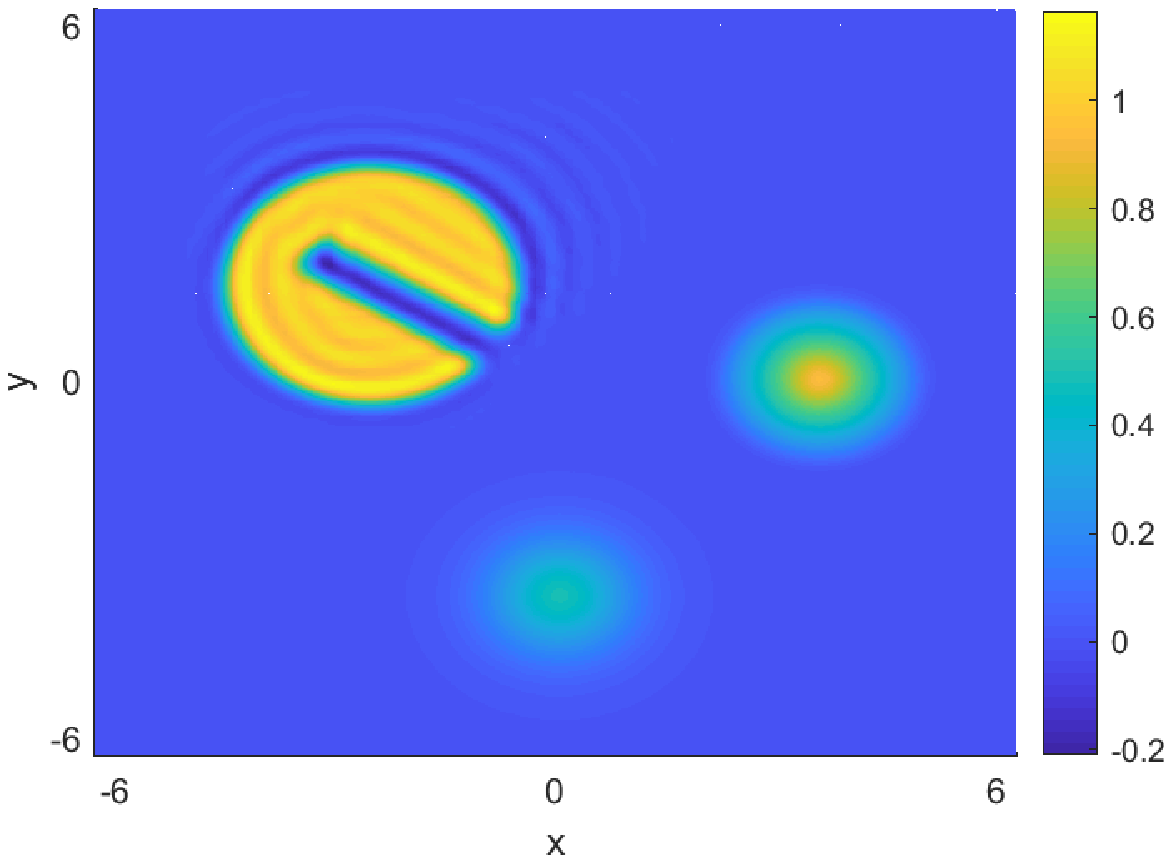}
\includegraphics[width=0.3\textwidth]{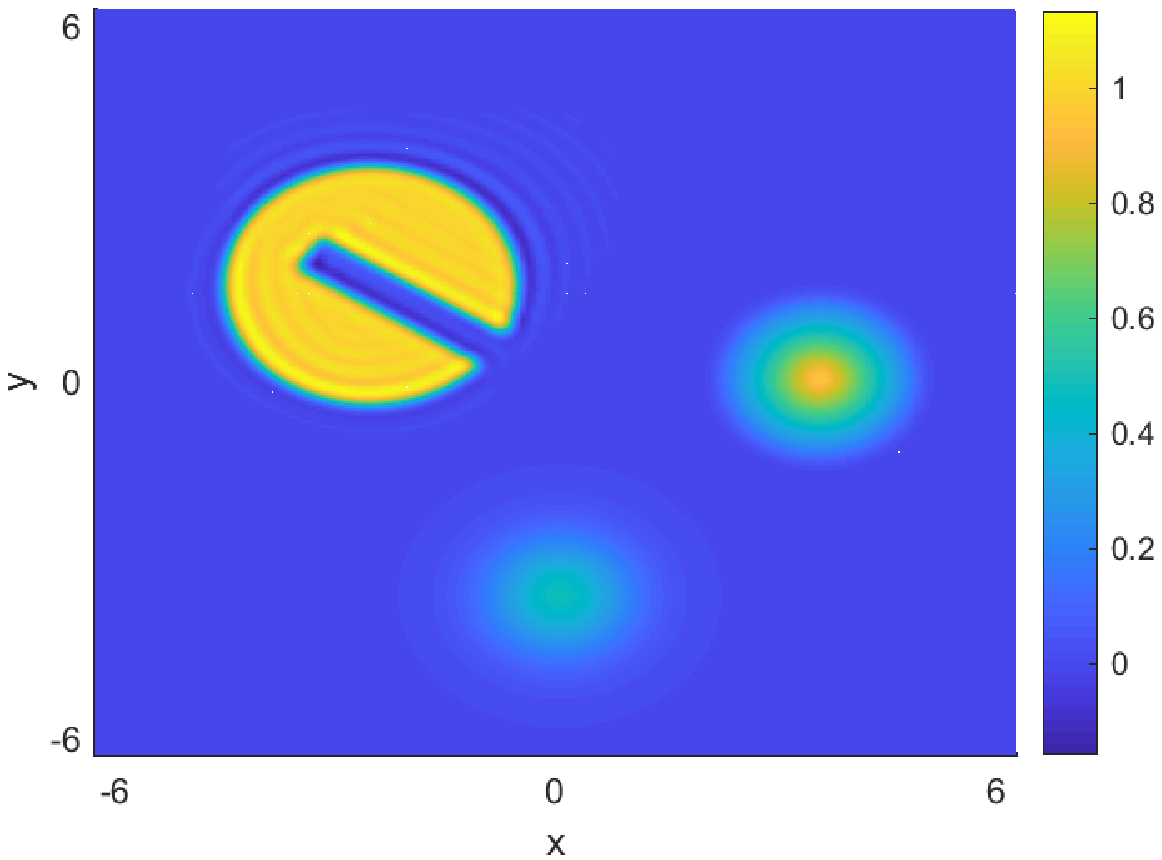}

\includegraphics[width=0.3\textwidth]{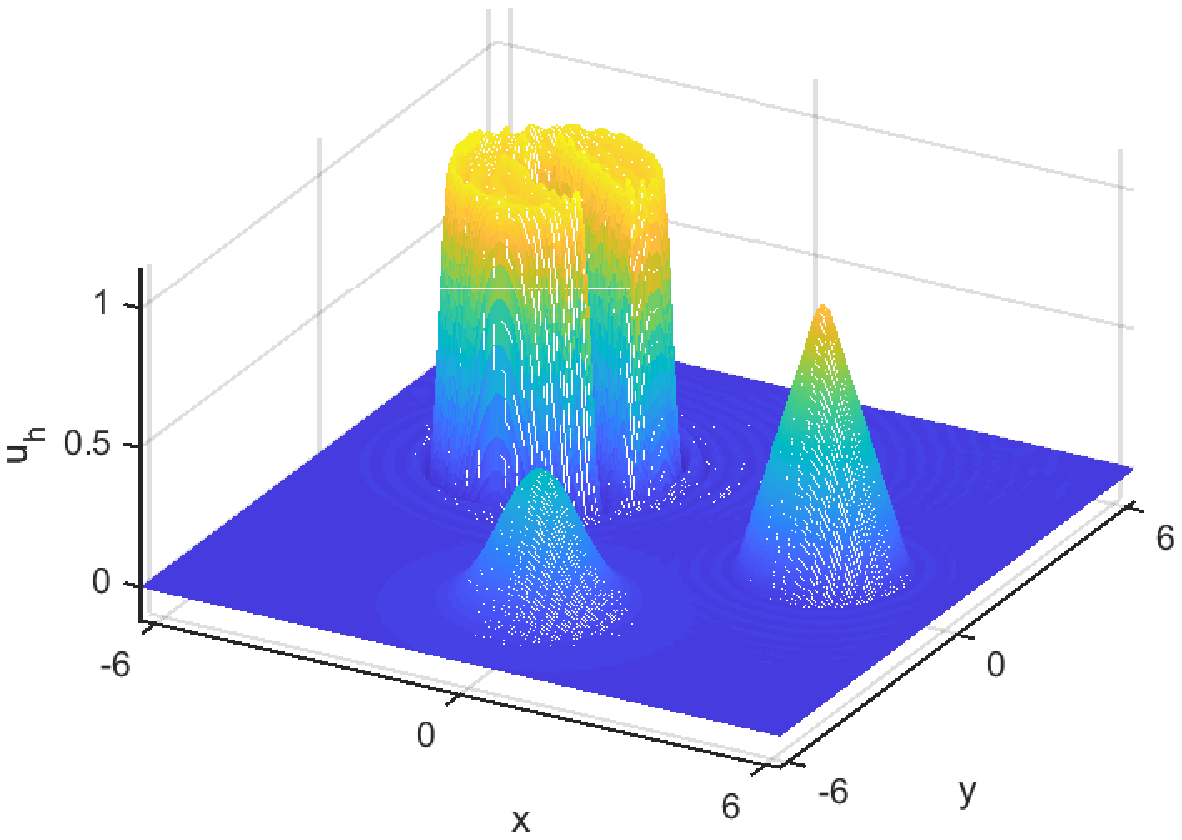}
\includegraphics[width=0.3\textwidth]{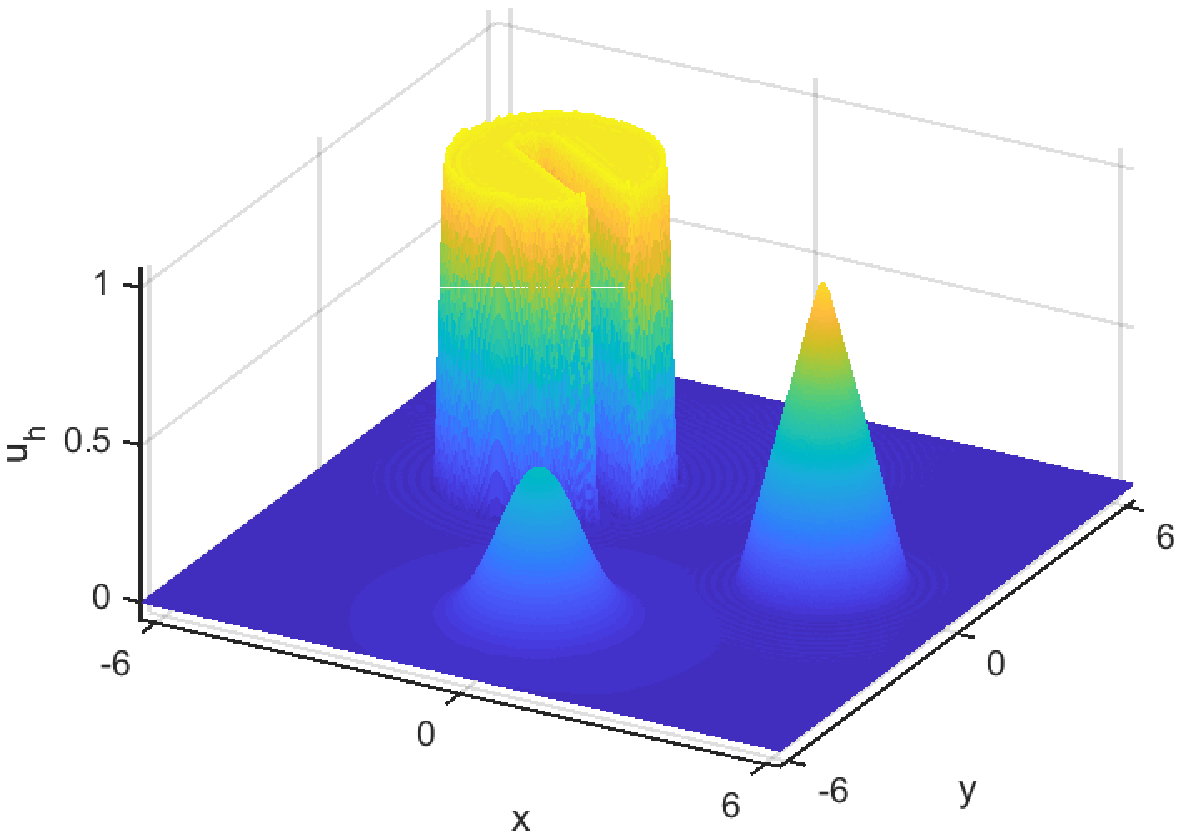}
\includegraphics[width=0.3\textwidth]{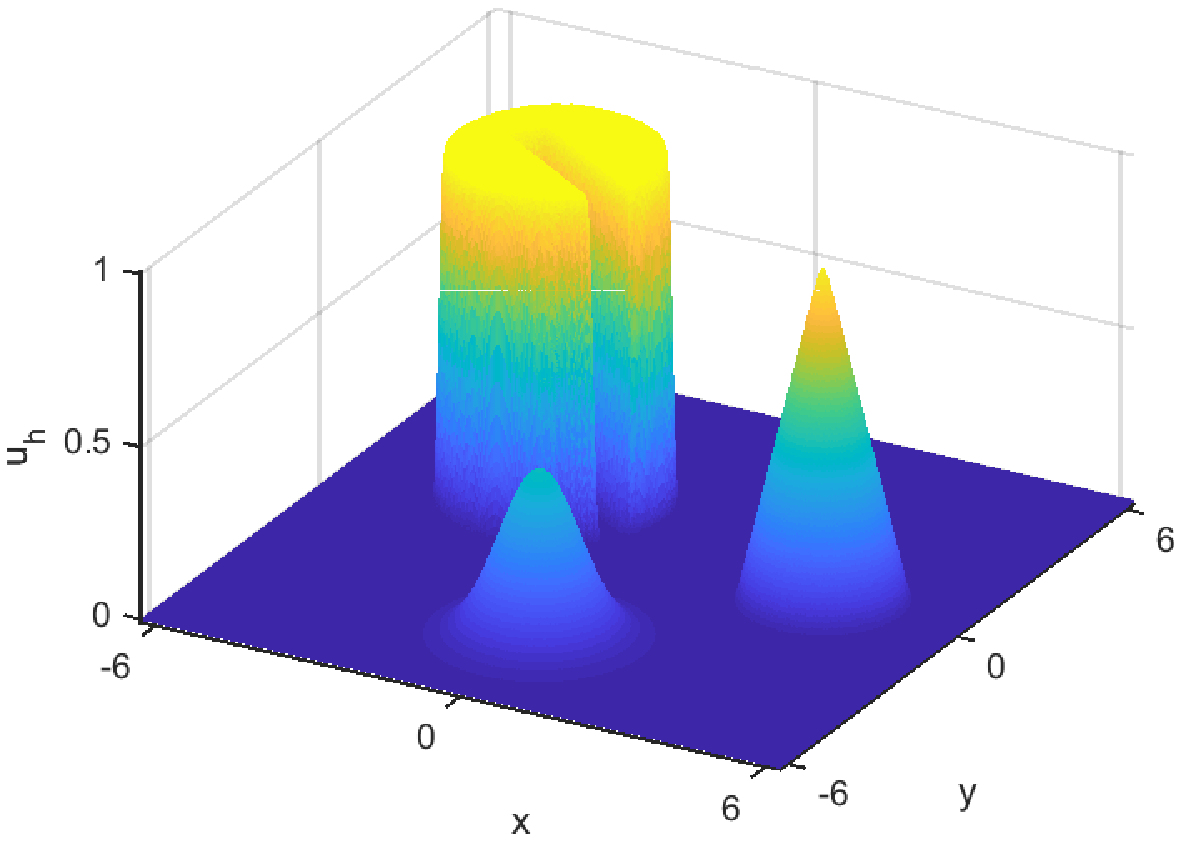}

\includegraphics[width=0.3\textwidth]{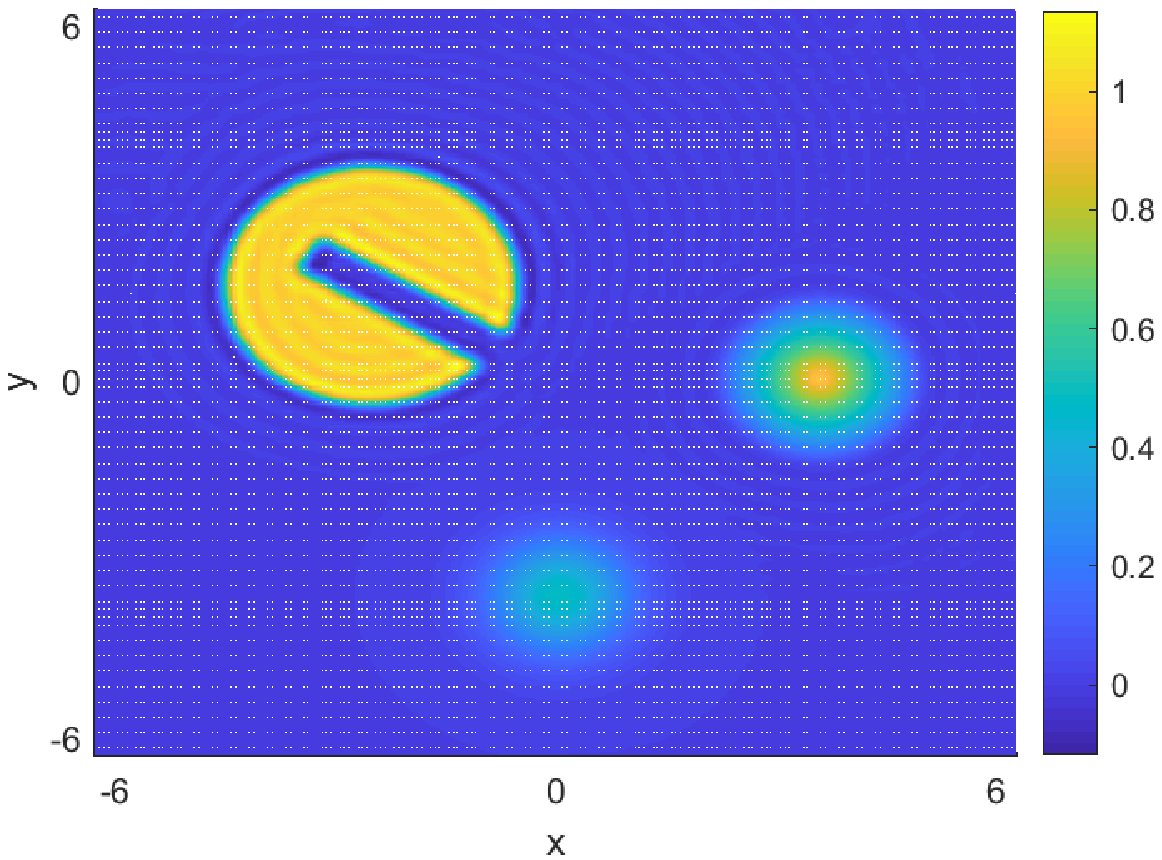}
\includegraphics[width=0.3\textwidth]{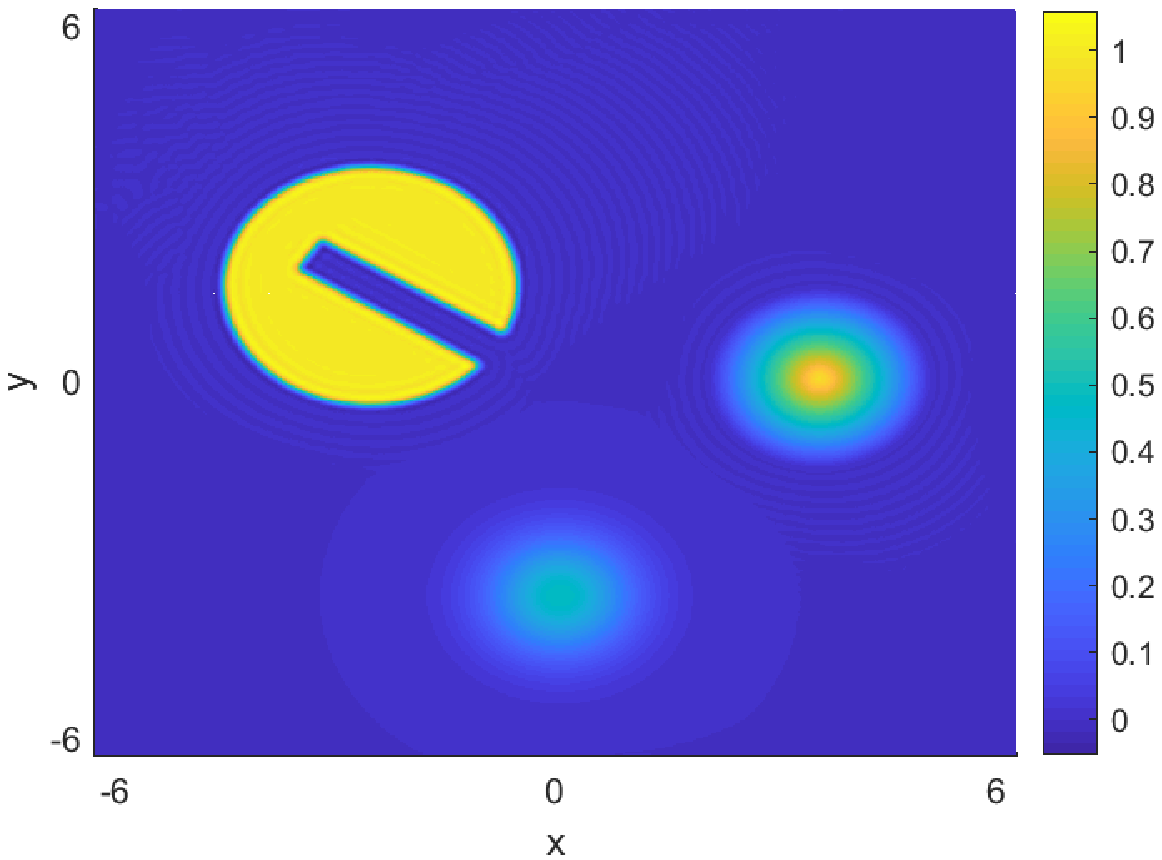}
\includegraphics[width=0.3\textwidth]{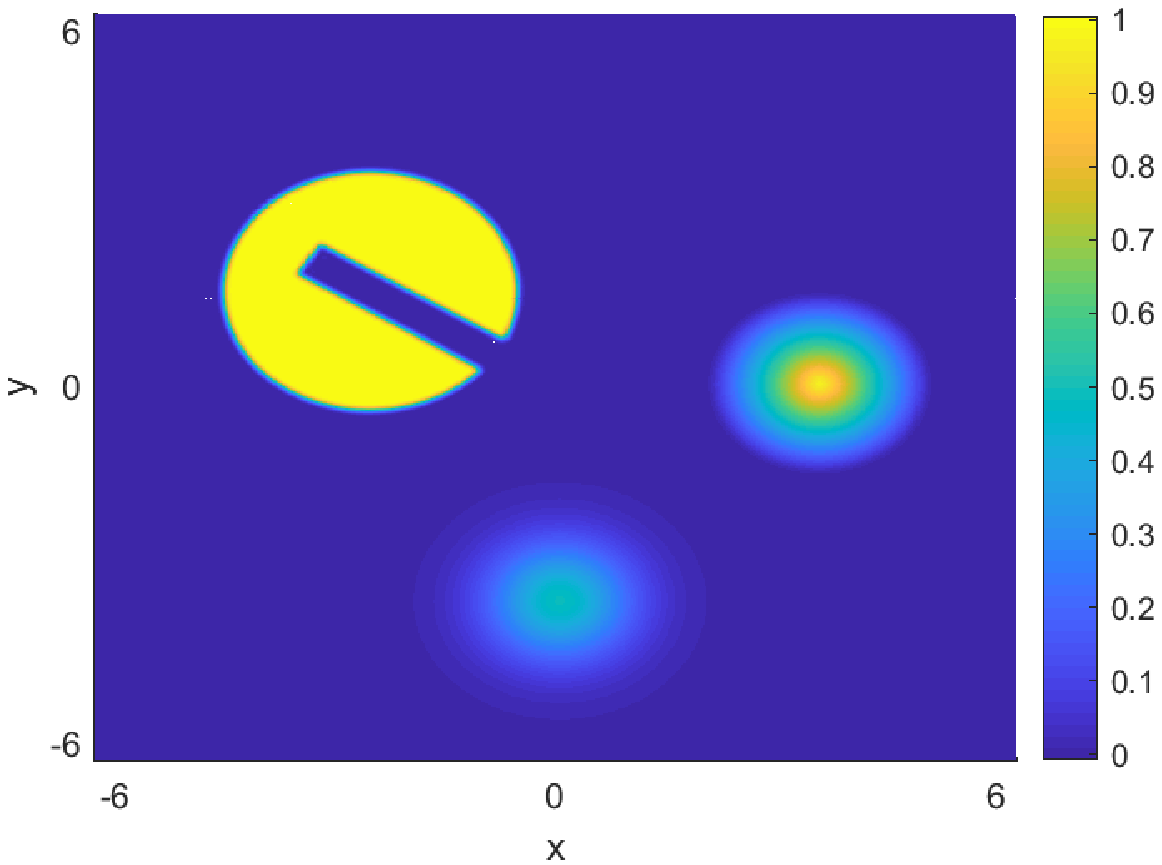}
\caption{The 3D and overlooking 2D image of the numerical solutions for 2D-Test 5. Top two rows: numerical solutions for the third-order method with $N = 50^2, 100^2, 150^2$ (from Left to Right). Bottom two rows: numerical solutions for the sixth-order method.
}
\label{fig:numerical solution 2DTest5}
\end{figure}

\section{Conclusion}\label{Sec:Conclusion}
In this paper, we have constructed the RDG space, which has the same order of accuracy as the standard DG space but uses fewer DoFs, where a narrow stencil reconstruction approach is proposed by using the high-order Legendre moments with the aim of maintaining the local property for the DG methods. Based on this approximation space, we apply the IMEX RK LDG method to the nonlinear CDR equation and present the error estimate in the $L^2$ norm. 
Several numerical experiments are presented to demonstrate the accuracy and performance of our method. The same idea can be applied to the irregular domain and partitions in our ongoing work.


\begin{appendix}
\setcounter{table}{0}
\renewcommand{\thetable}{\thesection.\arabic{table}}
\setcounter{equation}{0}
\renewcommand{\theequation}{\thesection.\arabic{equation}}
\setcounter{figure}{0}
\renewcommand{\thefigure}{\thesection.\arabic{figure}}

\section{The numerical analysis of Assumption \ref{admissible assumption}}\label{The numerical analysis of well-posedness}
Without loss of generality, consider any element $K_\theta\in \mathcal{T}_h$.
The stencil of $K_\theta$ introduced in Section \ref{sec:Compact reconstruction operator} is written as $S^d(K_\theta) = \{K_{\theta^1},K_{\theta^2},\cdots,K_{\theta^{s}}\}$, where $s = 3^d$ is the number of elements in the stencil. For convenience, we define the trace of a multiple indicator $\alpha = (\alpha_1,\cdots, \alpha_n)$ by $tr(\alpha) = \sum_{i = 1}^n\alpha_i$ and sort the set $A^k$ based on the ascendant sequence of the trace, i.e.,
\[
A^k = \{\alpha^1,\alpha^2,\cdots,\alpha^{(k+1)^d}\},
\]
where $tr(\alpha^1)\le tr(\alpha^2)\le,\cdots,\le tr(\alpha^{(k+1)^d})$.

As a result, we can express the approximation $R_{K_\theta}^ku$ as the following Legendre expansion
\[
R_{K_\theta}^ku(\pmb x) = \sum_{j = 1}^{(k+1)^d} u^{\alpha^j} L_{K_\theta}^{\alpha^j}(\pmb x),
\]
where $\pmb{u} = (u^{\alpha^1},u^{\alpha^2},\cdots,u^{\alpha^{(k+1)^d}})^T$ is the unknown coefficient vector to be determined.

For any element $K_{\theta^i}\in S^d(K_\theta)$, let $M_{K_{\theta^i}}^k$ denote an $(m+1)^d\times (k+1)^d$ matrix whose $j$th row is the Legendre moments vector obtained by applying the operator $I_{K_{\theta^i}}^{\alpha^j}$ on the Legendre basis vector $(L^{\alpha^1}_{K_\theta}(\pmb x),L^{\alpha^2}_{K_\theta}(\pmb x),\cdots,L^{\alpha^{(k+1)^d}}_{K_\theta}(\pmb x))$, i.e.,
\begin{equation*}
	m^k_{K_{\theta^i}} =
	\begin{pmatrix}
		I_{K_{\theta^i}}^{\alpha^1}L^{\alpha^1}_{K_\theta}(\pmb x)&I_{K_{\theta^i}}^{\alpha^1}L^{\alpha^2}_{K_\theta}(\pmb x)&\cdots&I_{K_{\theta^i}}^{\alpha^1}L^{\alpha^{(k+1)^d}}_{K_\theta}(\pmb x)\\
		I_{K_{\theta^i}}^{\alpha^2}L^{\alpha^1}_{K_\theta}(\pmb x)&I_{K_{\theta^i}}^{\alpha^2}L^{\alpha^2}_{K_\theta}(\pmb x)&\cdots&I_{K_{\theta^i}}^{\alpha^2}L^{\alpha^{(k+1)^d}}_{K_\theta}(\pmb x)\\
		\vdots&\vdots&&\vdots\\
		I_{K_{\theta^i}}^{\alpha^{(m+1)^d}}L^{\alpha^1}_{K_\theta}(\pmb x)&I_{K_{\theta^i}}^{\alpha^{(m+1)^d}}L^{\alpha^2}_{K_\theta}(\pmb x)&\cdots&I_{K_{\theta^i}}^{\alpha^{(m+1)^d}}L^{\alpha^{(k+1)^d}}_{K_\theta}(\pmb x)
	\end{pmatrix}.
\end{equation*}
Repeating the same process for every elements of the stencil $S^d(K_\theta)$, we can define the coefficient matrix of the linear system \eqref{admissible assumption equation} by
\begin{equation}\label{simple expressation of matrix}
M_{K_\theta}^k = \left(m_{K_{\theta^1}}^T,m_{K_{\theta^2}}^T,\cdots,m_{K_{\theta^s}}^T\right)^T.
\end{equation}
Ultimately, we obtain the matrix version of the system \eqref{admissible assumption equation} as follows:
\begin{equation}\label{matrix of local reconstruction operator}
	M_{K_\theta}^k\pmb{u} = \pmb{0},
\end{equation}
where $\pmb 0$ is the zeros vector.

For this system to have unique zero solution, there are two conditions to be simultaneously satisfied:
\begin{enumerate}
	\item  Matrix $M_{K_\theta}^k$ is a square matrix, i.e., $(k+1)/(m+1) = 3$.
	\item  The determinant of matrix $M_{K_\theta}^k$ must be non-zero.
\end{enumerate}
Under the first condition, let us discuss the second condition. To do that, we need to figure out each element of $M_{K_\theta}^k$. Take the element $I_{K_{\theta^i}}^{\alpha^r}L^{\alpha^l}_{K_\theta}(\pmb x)$ for example, we have
\begin{equation*}	
	\begin{aligned}
		I_{K_{\theta^i}}^{\alpha^r}L^{\alpha^l}_{K_\theta}(\pmb x) & = \left(\prod_{j = 1}^d\frac{2\alpha_j^r+1}{h_{K_{\theta^i}}^j}\right)\int_{K_{\theta^i}}L_{K_{\theta^i}}^{\alpha^r}(\pmb{x})L^{\alpha^l}_{K_\theta}(\pmb x)d\pmb{x}\\
		& = \prod_{j = 1}^d\left(\frac{2\alpha_j^r+1}{2}\int_{-1}^1\widehat{L}^{\alpha_j^r}(\widehat{x})L_{K_\theta}^{\alpha_j^l}\left(\frac{h_{K_{\theta^i}}^j}{2}\widehat{x}+x_{K_{\theta^i}}^j\right)d\widehat{x}\right)\\
		& =
		\prod_{j = 1}^d\left(\frac{2\alpha_j^r+1}{2}\int_{-1}^1\widehat{L}^{\alpha_j^r}(\widehat{x})\widehat{L}^{\alpha_j^l}\left(\frac{h_{K_{\theta^i}}^j}{h_{K_\theta}^j}\widehat{x}+\frac{2}{h_{K_\theta}^j}\left(x_{K_{\theta^i}}^j-x_{K_\theta}^j\right)\right)d\widehat{x}\right)\\
		& = \prod_{j = 1}^d\left(\frac{2\alpha_j^r+1}{2}\int_{-1}^1\widehat{L}^{\alpha_j^r}(\widehat{x})\widehat{L}^{\alpha_j^l}\left(a_{\theta^i}^j\widehat{x}+b_{\theta^i}^j\right)d\widehat{x}\right),
	\end{aligned}
\end{equation*}
where $a_{\theta^i}^j = \frac{h_{K_{\theta^i}}^j}{h_{K_\theta}^j}$ and $b_{\theta^i}^j = \frac{2}{h_{K_\theta}^j}(x_{K_{\theta^i}}^j-x_{K_\theta}^j)$.

Next, we will give numerical proof for different cases. Let us start with the one-dimensional case, i.e., $d = 1$. Where $\theta,\alpha$ are all just single indices. The element $K_\theta$ can be briefly written as $K_j$. Besides,
We have $\alpha^i = i-1$ for $i \ge 1$. For $k = 2$, we have $m = 0$ according to the first condition. When taking the center stencil $S_c^1(K_j) = \{K_{j-1},K_j,K_{j+1}\}$, we have
\begin{equation*}
	\begin{aligned}
		M^2_{K_j} &=
		\begin{pmatrix}
			I^0_{K_{j-1}}L^0_{K_j}(x)&I^0_{K_{j-1}}L^1_{K_j}(x)&I^0_{K_{j-1}}L^2_{K_j}(x) \\
			I^0_{K_{j}}L^0_{K}(x)&I^0_{K_{j}}L^1_{K_j}(x)&I^0_{K_{j}}L^2_{K_j}(x) \\
			I^0_{K_{j+1}}L^0_{K_j}(x)&I^0_{K_{j+1}}L^1_{K_j}(x)&I^0_{K_{j+1}}L^2_{K_j}(x) \\
		\end{pmatrix}\\
	&=
		\begin{pmatrix}
			1&b_{j-1}&\frac{3b_{j-1}^2+a_{j-1}^2-1}{2}\\
			1&0&0 \\
			1&b_{j+1}&\frac{3b_{j+1}^2+a_{j+1}^2-1}{2}\\
		\end{pmatrix}
		.
	\end{aligned}
\end{equation*}

Naturally, the determinant of $M_{K_j}^2$ can be computed as
\begin{equation*}
	\begin{aligned}
&\det(M_{K_j}^2)\\
=& \frac{3b_{j-1}^2b_{j+1} - 3b_{j-1}b_{j+1}^2 - b_{j-1}a_{j+1}^2 + b_{j-1} + b_{j+1}a_{j-1}^2 -b_{j+1}}{2}.
	\end{aligned}
\end{equation*}
To determine the value of $\det(M_{K_j}^2)$, we need more information. First, the regularity of mesh can induce the following condition 
\begin{equation}\label{mesh regular condition}
	0<a_{m}<a_{j-q}<a_{M},\ {\rm for}\ q = 2,1,-1,-2.
\end{equation}
Moreover, by simple calculation, there are $b_{j-1} = -(1+a_{j-1})$, $b_{j+1}=1+a_{j+1}$. With these informations, the following result can be directly obtained
\begin{equation*}	
	\begin{aligned}
		\det(M_{K_j}^2)|_{S_c^1} &= 2(a_{j+1} + 1)(a_{j-1} + 1)(a_{j-1} +a_{j+1} + 1)\\
		&>2(a_m+1)^2(2a_m+1).
	\end{aligned}
\end{equation*}
For the backward stencil $S_b^1(K_j) = \{K_{j-2},K_{j-1},K_{j}\}$, the similar determinant property results from the conditions $b_{j-2} = -(1+2a_{j-1}+a_{j-2}),\ b_{j-1}=-(1+a_{j-1})$ and \ref{mesh regular condition}, which is presented as follows,
\begin{equation*}	
	\begin{aligned}
		\det(M_{K_j}^2)|_{S_b^1} &= 2(a_{j-1}+1)(a_{j-1}+a_{j-2})(a_{j-1}+a_{j-2}+1)\\
		&>4a_m(a_m+1)(2a_m+1).
	\end{aligned}
\end{equation*}
In the same way, with the conditions $b_{j+1} = 1+a_{j+1},\ b_{j+2}=1+2a_{j+1}+a_{j+2}$ and \ref{mesh regular condition}, we can obtain the determinant property for the forward stencil $S_f^1(K_j) = \{K_j,K_{j+1},K_{j+2}\}$ as follows,
\begin{equation*}	
	\begin{aligned}
		\det(M_{K_j}^2)|_{S_f^1} &= 2(a_{j+1}+1)(a_{j+1}+a_{j+2})(a_{j+1}+a_{j+2}+1)\\
		&>4a_m(a_m+1)(2a_m+1).
	\end{aligned}
\end{equation*}

The determinant property for the case of $k = 5$ can be deduced following the proof process for $k = 2$. The results are presented in Table \ref{tab:nonzero determinant property_sixth order}.
\begin{table}[htb]
		\centering
		\renewcommand{\arraystretch}{1.5}
		\begin{tabular}{|c|c|}
			\hline
			stencil& $\det(M_{K_j}^5)$ \\
			\hline
			\multirow{2}{*}{$S_c^1(K_j)$}&$252a_{j-1}a_{j+1}(a_{j+1} + 1)^4 (a_{j-1} + 1)^4(a_{j-1} + a_{j+1} + 1)^3$\\
			&$>252a_m^2(a_m+1)^8(2a_m+1)^4$\\
			\hline
			\multirow{2}{*}{$S_b^1(K_j)$}&$252a_{j-2}a_{j-1}(a_{j-1} + 1)^4 (a_{j-2} + a_{j-1})^4(a_{j-2} + a_{j-1} + 1)^3$\\
			&$>4032a_m^6(a_m+1)^4(2a_m+1)^4$\\
			\hline
			\multirow{2}{*}{$S_f^1(K_j)$}&$252a_{j+1}a_{j+2}(a_{j+1} + 1)^4 (a_{j+1} + a_{j+2})^4(a_{j+1} + a_{j+2} + 1)^3$\\
			&$>4032a_m^6(a_m+1)^4(2a_m+1)^4$\\
			\hline
		\end{tabular}
	\caption{The determinant of $M_{K_j}^5$.}
	\label{tab:nonzero determinant property_sixth order}
\end{table}

Next, turn our attention to the two-dimensional case, i.e., $d = 2$. Here, $\theta,\alpha$ are all double index with the form $(i,j)$. The element $K_\theta$ can be briefly defined as $K_{i,j} = I_i^1\times I_j^2$. For $k = 2$, we have $m = 0$. We recall the definition of the two-dimensional stencil, $S^2(K_{i,j}) = S^1(I^1_i)\times S^1(I^2_j)$. Consequently, three different options for one-dimensional stencil produce nine different stencil strategies for two dimension, which are shown as follows,
\begin{small}
\begin{equation}
	\begin{aligned}
		&S_{cc}^2(K_{i,j}) = S_c^1(I^1_i)\times S_c^1(I^2_j),\ S_{cb}^2(K_{i,j}) = S_c^1(I^1_i)\times S_b^1(I^2_j),\ S_{cf}^2(K_{i,j}) = S_c^1(I^1_i)\times S_f^1(I^2_j),\\
		&S_{bb}^2(K_{i,j}) = S_b^1(I^1_i)\times S_b^1(I^2_j),\ S_{bc}^2(K_{i,j}) = S_b^1(I^1_i)\times S_c^1(I^2_j),\ S_{bf}^2(K_{i,j}) = S_b^1(I^1_i)\times S_f^1(I^2_j),\\
		&S_{ff}^2(K_{i,j}) = S_f^1(I^1_i)\times S_f^1(I^2_j),\ S_{fc}^2(K_{i,j}) = S_f^1(I^1_i)\times S_c^1(I^2_j),\ S_{fb}^2(K_{i,j}) = S_f^1(I^1_i)\times S_b^1(I^2_j).\\
	\end{aligned}
\end{equation}
\end{small}
Take the stencil $S_{cc}^2(K_{i,j})$ as an example, let us illustrate the proof process. We know that the basis functions of $Q^k(S^2(K_{i,j}))$ are the tensor product of the basis functions of $Q^k(S^1(I_i))$ and $Q^k(S^1(I_j))$. Considering the same product form of the stencils, the coefficient matrix can be expressed by the Kronecker product form
\[
M^2_{K_{i,j}} = M^2_{I_i}\otimes M^2_{I_j}.
\]
Given square matrices $A$ and $B$ with degrees $m$ and $n$, there is a well-known determinant conclusion for the Kronecker product \cite{henderson1983history, zhang2013kronecker},
\[
\det(A\otimes B) = \det(A)^n\det(B)^m.
\]
Therefore, we have
\[
\det(M^2_{K_{i,j}}) = \det(M^2_{I_i})^3\det(M^2_{I_j})^3.
\]
Based on the determinant property in one dimension, we can conclude that the same property also holds for every stencil strategies for two dimension. The same result can be obtained for $k = 5$ according to the fact that
\[
\det(M^5_{K_{i,j}}) = \det(M^5_{I_i})^6\det(M^5_{I_j})^6.
\]

\section{Proof of Theorem \ref{Theorem 1}}\label{Proof of Theorem1}
\begin{proof}
	The local reconstruction operator $R_{K}^k$ can be regarded as one interpolation operator. Thus, proving the $k$-exactness property \eqref{k-exactness property for local reconstruction operator} is equivalent to proving the uniqueness of polynomial interpolation problem which has been given by Assumption \ref{admissible assumption}.	
	
	With the $k$-exactness property \eqref{k-exactness property for local reconstruction operator}, the operator $R_{K}^k$ can be regarded as a projection operator which projects the Sobolev space $H^{k+1}(S(K))$ on the polynomial space $Q^k(S(K))$. Considering the $l_\infty$ norm $\Vert\cdot\Vert_{l_\infty(S(K))}$, as defined in Section \ref{Properties for Compact reconstruction operator}, we have the following property
	\begin{equation}\label{projection property}
		\Vert R_{K}^ku\Vert_{l_{\infty}(S(K))}=\Vert u\Vert_{l_{\infty}(S(K))}.
	\end{equation}
	
	Following \cite[Chapter 4.6]{brenner2008mathematical}, there exists an averaged Taylor polynomial $T^{k+1}u\in Q^{k}(S(K))$ such that
	\begin{equation*}
		\Vert u -T^{k+1}u\Vert_{0,\infty,S(K)}\le Ch^{k+1-\frac{d}{2}}.
	\end{equation*}
	With \eqref{norms equivalence}, \eqref{projection property}, and \eqref{upper bound for l_infty discrete norm}, we have
	\begin{equation*}
		\begin{aligned}
			\Vert T^{k+1}u -R_{K}^ku\Vert_{0,\infty,K}&\le C\Vert T^{k+1}u -R_{K}^ku\Vert_{0,\infty,S(K)}=C\Vert R_{K}^k(T^{k+1}u -u)\Vert_{0,\infty,S(K)}\\
			&\le C\Vert R_{K}^k(T^{k+1}u -u)\Vert_{l_{\infty}(S(K))}= C\Vert T^{k+1}u -u\Vert_{l_{\infty}(S(K))}\\
			&\le C\Vert T^{k+1}u -u\Vert_{0,\infty,S(K)}\le Ch^{k+1-\frac{d}{2}}.
		\end{aligned}
	\end{equation*}
	Therefore,
	\begin{equation*}
		\Vert u-R_{K}^ku\Vert_{0,\infty,K}\le \Vert u-T^{k+1}u\Vert_{0,\infty,S(K)}+\Vert T^{k+1}u-R_{K}^ku\Vert_{0,\infty,S(K)}\le Ch^{k+1-\frac{d}{2}},
	\end{equation*}
	which gives \eqref{l_infty error estimate for local reconstruction operator}.
	
	Following \eqref{l_infty error estimate for local reconstruction operator}, the $L^2$ error estimate \eqref{l2 error estimate for local reconstruction operator} can be derived directly as follows,
	\[
	\Vert u_h -R_{K}^ku\Vert_{0,K}\le Ch^{\frac{d}{2}}\Vert u_h -R_{K}^ku\Vert_{0,\infty,K}\le Ch^{k+1}.
	\]
	
	By the approximation property \eqref{Approximation property}, we can choose an approximation polynomial $u_h\in Q^k(S(K)$ such that
	\[
	\Vert u -u_h\Vert_{0,K}+h| u -u_h|_{1,K}\le Ch^{k+1}.
	\]
With \eqref{Approximation property}, \eqref{l2 error estimate for local reconstruction operator}, and the inverse property \eqref{inverse properties},	we can derive \eqref{sime norm error estimate for local reconstruction operator} smoothly
	\begin{equation*}
		\begin{aligned}
			| u -R_{K}^ku|_{1,K}&\le | u -u_h|_{1,K}+| u_h -R_{K}^ku|_{1,K}\\
			&\le Ch^k+Ch^{-1}\Vert u_h -R_{K}^ku\Vert_{0,K}\\
			&\le Ch^k+Ch^{-1}\Vert u -u_h\Vert_{0,K}+Ch^{-1}\Vert u -R_{K}^ku\Vert_{0,K}\\
			&\le Ch^k.
		\end{aligned}
	\end{equation*}
	Here, the proof is completed.
\end{proof}
\section{Proof of Theorem \ref{l2 estimate for the LDG}}\label{Proof of Theorem3}
\begin{proof}
Given that the exact solution $u$ and $\pmb q$ also satisfies the global weak formulation \eqref{global weak formulation(gWF)}, we obtain the following error equation by a simple subtraction
\begin{equation*}
\begin{aligned}
&((u-u_h)_t,v_h)+\sum_{i=1}^d(q^i-q_h^i,p_h^i)\\
&=\mathcal{H}(u,v_h)-\mathcal{H}(u_h,v_h)+\mathcal{L}(\pmb q-\pmb q_h,v_h)+\sum_{i=1}^d\mathcal{K}^i(u-u_h,p_h^i)+\mathcal{R}(u,v_h)-\mathcal{R}(u_h,v_h).
	\end{aligned}
\end{equation*}
Here, we denote
\begin{equation*}
\begin{aligned}
\mathcal{K}(u-u_h,\pmb p_h)& =\sum_{i=1}^d\mathcal{K}^i(u-u_h,p_h^i)\\
&= -\sqrt{\varepsilon}\sum_{i=1}^d\left((u-u_h,(p_h^i)_{x_i})_K-\langle u-\hat{u}_h, p_h^in_i\rangle_{\partial K}\right)\\
&=-\sqrt{\varepsilon}\left((u-u_h,\nabla\cdot \pmb p_h)_K-\langle u-\hat{u}_h, \pmb p_h\cdot \pmb n\rangle_{\partial K}\right).
	\end{aligned}
\end{equation*}
Thus a more neat error equation can be described as follows,
\begin{equation}\label{error equation}
\begin{aligned}
((u-u_h)_t,v_h)+(\pmb q-\pmb q_h,\pmb p_h)=&\mathcal{H}(u,v_h)-\mathcal{H}(u_h,v_h)+\mathcal{L}(\pmb q-\pmb q_h,v_h)+\mathcal{K}(u-u_h,\pmb p_h)\\
&+\mathcal{R}(u-u_h,v_h).
	\end{aligned}
\end{equation}
Denote
\begin{equation}\label{projection error}
\xi = R^ku-u_h,\ \xi^e = R^ku-u,\ \pmb\eta = R^k\pmb q-\pmb q_h,\ \pmb\eta^e = R^k\pmb q-\pmb q,\
\end{equation}
and take the test function as
\begin{equation}\label{test function}
v_h = \xi,\ \pmb p_h = \pmb\eta.
\end{equation}
We obtain the energy equality
 \begin{equation}\label{energy equality}
((\xi-\xi^e)_t,\xi)+(\pmb\eta-\pmb\eta^e,\pmb\eta)=\mathcal{H}(u,\xi)-\mathcal{H}(u_h,\xi)+\mathcal{L}(\pmb q-\pmb q_h,\xi)+\mathcal{K}(u-u_h,\pmb\eta)+\mathcal{R}(u-u_h,\xi).
\end{equation}
Next we will estimate each term on the right-hand side of the energy equality.

To estimate the nonlinear convection term $\mathcal{H}(u,\xi)-\mathcal{H}(u_h,\xi)$, we need to make $a\ priori$ assumption, for small enough $h$, we have
\begin{equation}\label{a priori assumption}
\Vert u-u_h\Vert_{0}\le Ch^{\frac{d+1}{2}}.
\end{equation}
With this assumption, we have $\Vert u-u_h\Vert_{0,\infty}\le Ch^{\frac{1}{2}}$. Moreover, the property \eqref{boundary norm and infty norm estimate for global reconstruction operator} implies that $\Vert R^ku-u_h\Vert_{0,\infty}\le Ch^{\frac{1}{2}}$.

Here, we have
 \begin{equation*}
 \begin{aligned}
&\mathcal{H}(u,\xi)-\mathcal{H}(u_h,\xi) \\
= &\sum_{K\in \mathcal{T}_h}\left((\pmb b(f(u)-f(u_h)),\nabla \xi)_K-\langle\pmb b(f(u)-\hat{f}(u_h))\cdot \pmb n,\xi\rangle_{\partial K}\right)\\
=&(\pmb b(f(u)-f(u_h)),\nabla\cdot\xi)-\langle \pmb b(f(u)-\hat{f}(u_h)),[\xi]\rangle\\
=&(\pmb b(f(u)-f(u_h)),\nabla\cdot\xi)-\langle \pmb b(f(u)-f(\bar{u}_h)),[\xi]\rangle+\langle \pmb b(f(\bar{u}_h)-\hat{f}(u_h)),[\xi]\rangle\\
=&I+II.
\end{aligned}
\end{equation*}
According to \cite[Lemma 3.4]{xu2007error}, the second part can be estimated as
 \begin{equation*}
II = \langle \pmb b(f(\bar{u}_h)-\hat{f}(u_h)),[\xi]\rangle\le -\frac{3}{4}\Vert\pmb b\Vert_{1,\infty}\langle \alpha(\hat{f};u_h),[\xi]^2\rangle+Ch^{2k+1}.
 \end{equation*}
Where $\alpha(\hat{f};u_h)$ is non-negative and bounded, which is defined in \cite{zhang2004error}.
Next, we estimate the first part $I$. Follow \cite[Appendix A.1]{xu2007error}, $I$ can be divided into six parts as follows,
\begin{equation*}
 \begin{aligned}
I_1 &= (\pmb bf'(u)\xi,\nabla\cdot\xi)+\langle \pmb bf'(u)\bar{\xi},[\xi]\rangle,\\
I_2 &= \frac{1}{2}\left((\pmb bf''(u)\xi^2,\nabla\cdot\xi)+\langle \pmb bf''(u)\bar{\xi}^2,[\xi]\rangle\right),\\
I_3 &= -\left((\pmb bf'(u)\xi^e,\nabla\cdot\xi)+\langle \pmb bf'(u)\bar{\xi}^e,[\xi]\rangle\right),\\
I_4 &= (\pmb bf''(u)\xi^e\xi,\nabla\cdot\xi)+\langle \pmb bf''(u)\bar{\xi}^e\bar{\xi},[\xi]\rangle,\\
I_5 &= -\frac{1}{2}\left((\pmb bf''(u)(\xi^e)^2,\nabla\cdot\xi)+\langle \pmb bf''(u)(\bar{\xi}^e)^2,[\xi]\rangle\right),\\
I_6 &= \frac{1}{6}\left((\pmb bf_u'''(\xi-\xi^e)^3,\nabla\cdot\xi)+\langle \pmb b\tilde{f}_u'''(\bar{\xi}-\bar{\xi}^e)^3,[\xi]\rangle\right),
\end{aligned}
\end{equation*}
where $f_u'''$ and $\tilde{f}_u'''$ are the factors in the remainder of Taylor expansion of $f(u_h)$ and $f(\bar{u}_h)$ separately. By integration by parts, we can estimate each term now:
\begin{itemize}
\item $I_1$ term\\
\begin{equation*}
I_1 = \frac{1}{2}(\nabla \cdot (\pmb bf'(u)),\xi^2)\le C\Vert\xi\Vert_{0}^2.
\end{equation*}
\item $I_2$ term\\
\begin{equation*}
I_2 = \frac{1}{6}(\nabla \cdot(\pmb b f''(u))\xi,\xi^2)+\frac{1}{24}\langle \pmb bf''(u)[\xi],[\xi]^2\rangle.\\
\end{equation*}
By a Taylor expansion, we have
\begin{equation*}
 \begin{aligned}
f''(u)[\xi] &= f''(u_h)[\xi] +f''_u(u-u_h)[\xi]\\
&=-f''(u_h)[u_h]+f''(u_h)[\xi^e]+f''_u(u-u_h)[\xi]\\
&\le 8\alpha(\hat{f};u_h)+8|[u-u_h]|^2+f''(u_h)[\xi^e]+f''_u(u-u_h)[\xi].
\end{aligned}
\end{equation*}
Thus,
\begin{equation*}
 \begin{aligned}
I_2 \le& \frac{1}{3}\Vert \pmb b\Vert_{1,\infty}\langle \alpha(\hat{f};u_h),[\xi]^2\rangle\\
&+C\left(\Vert\xi\Vert_{0,\infty}+h^{-1}(\Vert u-u_h\Vert_{0,\infty}^2+\Vert\xi^e\Vert_{0,\mathcal{E}_h}^2+\Vert\xi\Vert_{0,\mathcal{E}_h}^2)\right)\Vert\xi\Vert_{0}\\
\le&\frac{1}{3}\Vert \pmb b\Vert_{1,\infty}\langle \alpha(\hat{f};u_h),[\xi]^2\rangle+C\Vert\xi\Vert_{0}.
\end{aligned}
\end{equation*}
\item $I_3$ term\\
\begin{equation*}
 \begin{aligned}
I_3 &=
-(\pmb bf'(u)\xi^e,\nabla\cdot\xi)-\langle \pmb b(f'(u)-f'(\bar{u}_h))\bar{\xi}^e,[\xi]\rangle-\langle \pmb bf'(\bar{u}_h)\bar{\xi}^e,[\xi]\rangle\\
&\le C \Vert\xi\Vert_{0} +2\langle\pmb b (\alpha(\hat{f};u_h)+C|[u-u_h]|)\bar{\xi}^e,[\xi]\rangle+Ch^{2k}\\
&\le \frac{1}{6}\Vert \pmb b\Vert_{1,\infty}\langle \alpha(\hat{f};u_h),[\xi]^2\rangle+C\Vert\xi\Vert_{0}+Ch^{2k}.
\end{aligned}
\end{equation*}
\item $I_4$, $I_5$, and $I_6$ terms\\
\begin{equation*}
 \begin{aligned}
I_4 &\le Ch^{-1}\Vert\xi^e\Vert_{0,\infty}\Vert\xi\Vert_{0}\le C\Vert\xi\Vert_{0},\\
I_5 &\le C\Vert\xi^e\Vert_{0,\infty}(h^{-1}\Vert\xi^e\Vert_{0}+h^{-\frac{1}{2}}\Vert\xi^e\Vert_{0,\mathcal{E}_h})\Vert\xi\Vert_{0}\le C\Vert\xi\Vert_{0}+Ch^{2k+2},\\
I_6 &\le C\Vert u-u_h\Vert_{0,\infty}(h^{-1}(\Vert\xi\Vert_{0}+\Vert\xi^e\Vert_{0})+h^{-\frac{1}{2}}\Vert\xi^e\Vert_{0,\mathcal{E}_h})\Vert\xi\Vert_{0}\\
&\le C\Vert\xi\Vert_{0}+Ch^{2k+1}.
\end{aligned}
\end{equation*}
\end{itemize}
Combing with all the above estimates, we can derive that
 \begin{equation*}
I \le \frac{1}{2}\Vert \pmb b\Vert_{1,\infty}\langle \alpha(\hat{f};u_h),[\xi]^2\rangle+C\Vert\xi\Vert_{0}+Ch^{2k}.
 \end{equation*}
Naturally, we have the following estimate
  \begin{equation*}
\mathcal{H}(u,\xi)-\mathcal{H}(u_h,\xi) \le -\frac{1}{4}\Vert \pmb b\Vert_{1,\infty}\langle \alpha(\hat{f};u_h),[\xi]^2\rangle+C\Vert\xi\Vert_{0}+Ch^{2k}.
\end{equation*}

Then we focus on the diffusion part. We have
\begin{equation*}
 \begin{aligned}
&\mathcal{L}(\pmb q-\pmb q_h,\xi)+\mathcal{K}(u-u_h,\pmb\eta)\\
=&-\sqrt{\varepsilon}\left((\pmb q-\pmb q_h,\nabla \xi)+(u-u_h,\nabla\cdot \pmb \eta)-\langle \pmb q-\hat{\pmb q}_h\cdot \pmb n^+,[\xi]\rangle+\langle u-\hat{u}_h, [\pmb \eta]\cdot \pmb n^+\rangle\right)\\
=&-\sqrt{\varepsilon}\left((\pmb \eta-\pmb \eta^e,\nabla \xi)+(\xi-\xi^e,\nabla\cdot \pmb \eta)-\langle \hat{\pmb \eta}-\hat{\pmb \eta}^e\cdot \pmb n^+,[\xi]\rangle+\langle \hat{\xi}-\hat{\xi}^e, [\pmb \eta]\cdot \pmb n^+\rangle\right)\\
=&\sqrt{\varepsilon}\left((\pmb \eta^e,\nabla \xi)+(\xi^e,\nabla\cdot \pmb \eta)-\langle \hat{\pmb \eta}^e\cdot \pmb n^+,[\xi]\rangle+\langle \hat{\xi}^e, [\pmb \eta]\cdot \pmb n^+\rangle\right)\\
\le& \frac{1}{2}\Vert\pmb \eta\Vert_{0}+C\Vert\xi\Vert_{0}+Ch^{2k}.
\end{aligned}
\end{equation*}

Finally, we estimate the reaction part as follows,
  \begin{equation*}
\mathcal{R}(u,\xi)-\mathcal{R}(u_h,\xi)=-(r(u)-r(u_h),\xi)=-(r'_u(\xi-\xi^e),\xi)\le C\Vert\xi\Vert_{0}+Ch^{2k+2}.
\end{equation*}
With the above estimates and Young's inequality, the energy equation becomes
  \begin{equation*}
 \begin{aligned}
\frac{1}{2}\left(\frac{d\Vert\xi\Vert_0^2}{dt}+\Vert\pmb\eta\Vert_0^2\right)+\frac{1}{4}\Vert \pmb b\Vert_{1,\infty}\langle \alpha(\hat{f};u_h),[\xi]^2\rangle&\le (\xi^e_t,\xi)+(\pmb \eta^e,\pmb \eta)+C\Vert\xi\Vert_{0}+Ch^{2k}\\
&\le C\Vert\xi\Vert_{0}+\frac{1}{4}\Vert\pmb \eta\Vert_{0}+Ch^{2k}.
\end{aligned}
\end{equation*}
Thus,
\begin{equation*}
\frac{1}{2}\frac{d\Vert\xi\Vert_0^2}{dt}+\frac{1}{4}\Vert\pmb\eta\Vert_0^2+\frac{1}{4}\Vert \pmb b\Vert_{1,\infty}\langle \alpha(\hat{f};u_h),[\xi]^2\rangle\le C\Vert\xi\Vert_{0}+Ch^{2k}.
\end{equation*}
By the Gronwall inequality and triangle inequality, the proof is completed as follows,
\begin{equation*}
\Vert u-u_h\Vert_0^2+\int_{0}^T\Vert\pmb q-\pmb q_h\Vert_0^2(t)dt\le Ch^{2k}.
\end{equation*}
\end{proof}
\end{appendix}

\bibliographystyle{abbrv}
\bibliography{ref}
\end{document}